\documentclass{amsart}
\usepackage{amssymb}

\newtheorem{theorem}{Theorem}
\newtheorem{corollary}[theorem]{Corollary}

\newtheorem{lemma}[theorem]{Lemma}
\newtheorem{proposition}[theorem]{Proposition}
\newtheorem{non-theorem}{Non-Theorem}

\theoremstyle{remark}

\numberwithin{equation}{section}
\numberwithin{theorem}{section}

\newcommand\paperbody%
        {}

\newcommand{\df}[1]{\mathfrak{#1}}

\renewcommand{\tilde}{\widetilde}
\renewcommand{\bar}{\overline}
\renewcommand{\hat}[1]{\widehat{#1}}

\newcommand\wt{\widetilde}

\newcommand{\rest}[1]{\big\rvert_{#1}} 

\newcommand{\dzero}[1]{\frac{\partial}{\partial #1}\biggr\rvert_{#1 =0}}
\newcommand\pa{\partial}
\newcommand{\owedge}{\wedge}
\newcommand{\tm}{\operatorname{tm}}

\newcommand\eps\varepsilon

\newcommand\fbsc{\operatorname{b,sc}}

\newcommand\CI{{\mathcal{C}}^{\infty}}

\newcommand{\lrpar}[1]{\left( #1 \right)}
\newcommand{\lrspar}[1]{\left[ #1 \right]}
\newcommand\ang[1]{\langle #1 \rangle}
\newcommand{\lrbrac}[1]{\lbrace #1 \rbrace}

\DeclareMathOperator*{\FP}{\operatorname{FP}}
\DeclareMathOperator*{\Res}{\operatorname{Res}}
\DeclareMathOperator*{\RN}{\operatorname{RN}}

\renewcommand\Re{\operatorname{Re}}

\newcommand\diag{\mathrm{diag}}

\newcommand\Id{\operatorname{Id}}
\newcommand{\coker}{\mathrm{ coker }}
\newcommand\ind{\operatorname{ind}}
\newcommand{\Ind}{\mathrm{Ind}}

\newcommand\Ch{\operatorname{Ch}}

\newcommand\tCh{\operatorname{\widetilde{Ch}}}

\newcommand\tr{\operatorname{tr}}

\newcommand\Tr{\operatorname{Tr}}

\newcommand\RTr{{}^R\operatorname{Tr}} 
\newcommand\Trsig[2][{}]{\widehat{\operatorname{Tr}^\sigma_{#1}}\left( #2 \right)} 
\newcommand\Trpa[2][{}]{\widehat{\operatorname{Tr}^\partial_{#1}}\left( #2 \right)}
\newcommand\Trpasig[2][{}]{\operatorname{Tr}^{\partial,\sigma}_{#1}\left( #2 \right)}

\newcommand\tTrsig{\widehat{\operatorname{Tr}^\sigma}} 
\newcommand\tTrpa{\widehat{\operatorname{Tr}^\partial}}
\newcommand\tTrpasig{\operatorname{Tr}^{\partial,\sigma}}

\newcommand{\bbC}{\mathbb{C}}

\newcommand{\bbE}{\mathbb{E}}
\newcommand{\bbF}{\mathbb{F}}

\newcommand{\bbQ}{\mathbb{Q}}
\newcommand{\bbR}{\mathbb{R}}
\newcommand{\bbS}{\mathbb{S}}

\newcommand{\bbZ}{\mathbb{Z}}

\newcommand\cA{\mathcal{A}}

\newcommand\cC{\mathcal{C}}

\newcommand\cF{\mathcal{F}}
\newcommand\cG{\mathcal{G}}
\newcommand\cH{\mathcal{H}}
\newcommand\cI{\mathcal{I}}

\newcommand\cK{\mathcal{K}}

\newcommand\cN{\mathcal{N}}
\newcommand\cO{\mathcal{O}}

\newcommand\cT{\mathcal{T}}
\newcommand\cU{\mathcal{U}}
\newcommand\cV{\mathcal{V}}

\newcommand\cZ{\mathcal{Z}}

\newcommand\tcC{\tilde{\mathcal{C}}}
\newcommand\tcH{\tilde{\mathcal{H}}}

\newcommand\datver[1]{\def\datverp%
 {\par\boxed{\boxed{\text{#1; Run: \today}}}}}

\usepackage[all]{xy}

\newcommand\comp{\overline}
\newcommand\scat{\operatorname{sc}}
\newcommand\Kc{K_{\text{c}}}
\newcommand\Hc{H_{\text{c}}}
\newcommand\Hce{H^{\even}_{\text{c}}}

\newcommand\Kco{K^1_\text{c}}
\newcommand\bo{\operatorname{b}}
\newcommand\even{\text{even}}
\newcommand\odd{\text{odd}}
\newcommand\ZS{{}^{0}S}
\newcommand\ZT{{}^{0}T}
\newcommand\inside[1]{#1^{\circ}}
\newcommand\scT{\operatorname{{}^{\scat}T}}
\newcommand\scS{\operatorname{{}^{\scat}S}}
\newcommand\ZPs[2]{\Psi^{#1 }_{0}(#2)}
\newcommand\bTr{\overline{\operatorname{Tr}}}
\newcommand\Td{\operatorname{Td}}
\newcommand\td{\tilde{d}}

\newcommand\ie{i\@.e\@. }
\newcommand\eg{e\@.g\@. }

\newcommand\CIc{{\mathcal{C}}^{\infty}_c}

\newcommand\BPs[2]{\Psi_{\text{b}}^{#1}({#2})} 

\newcommand\Mand{\text{ and }}
\newcommand\Mst{\text{ s.t. }}
\newcommand\Mso{\text{ so }}
\newcommand\Mwhere{\text{ where }}
\newcommand\Mwith{\text{ with }}

\datver{1.0; Revised: 12-10-2008}

\begin{document}
\title
[Relative Chern character, boundaries and index formul\ae]
{Relative Chern character, boundaries and index formul\ae}

\author{Pierre Albin}
\author{Richard Melrose}
\address{Department of Mathematics, Massachusetts Institute of Technology}
\email{pierre@math.mit.edu}
\email{rbm@math.mit.edu}
\thanks{The first author was partially supported by an NSF postdoctoral
  fellowship and the second author received 
  partial support under NSF grant DMS-0408993.}

\begin{abstract} For three classes of elliptic pseudodifferential operators
on a compact manifold with boundary which have `geometric K-theory', namely
the `transmission algebra' introduced by Boutet de Monvel
\cite{BoutetdeMonvel1}, the `zero algebra' introduced by Mazzeo in
\cite{Mazzeo:Hodge,Mazzeo-Melrose} and the `scattering algebra' from
\cite{MR95k:58168}, we give explicit formul\ae\ for the Chern character of
the index bundle in terms of the symbols (including normal operators at the
boundary) of a Fredholm family of fibre operators. This involves
appropriate descriptions, in each case, of the cohomology with compact
supports in the interior of the total space of a vector bundle over a
manifold with boundary in which the Chern character, mapping from the
corresponding realization of K-theory, naturally takes values.
\end{abstract}

\maketitle

\tableofcontents

\section*{Introduction}

Among the different algebras of pseudodifferential operators on a compact
manifold with boundary, those for which the stable homotopy classes of the
Fredholm elements reduce to the relative K-theory of the cotangent bundle
can be expected to have the simplest, most local, index formul\ae. These
cases include the algebra of transmission operators introduced by Boutet de
Monvel, for which this is shown in essence in \cite{BoutetdeMonvel1}, the
algebra of `zero pseudodifferential operators' of \cite{Mazzeo-Melrose1}
and the algebra of `scattering pseudodifferential operators' of
\cite{MR95k:58168}. For a Fredholm family of such operators defined on the
fibers of a fibration,
\begin{equation}
\xymatrix{Z\ar@{-}[r]&M\ar[d]^{\psi}\\ &B,}
\label{rccbif.3}\end{equation}
the families index, composed with the Chern character, therefore gives a
map (the same in all three cases)
\begin{equation}
K^0\lrpar{T^*(\inside{M}/B)}\overset{\ind}\longrightarrow K(B)
\overset{\Ch}\longrightarrow H^{\even}(B),\ \inside{M}=M\setminus\pa M.
\label{rccbif.1}\end{equation}
This map is always given by the general Atiyah-Singer formula (including here the
extension by Atiyah and Bott) just as in the boundaryless case in
\cite{MR38:5243,MR38:5245}
\begin{equation}
\Ch(\Ind(A))=\psi_!\left(\Ch([A])\Td(Z)\right)
\label{rccbif.2}\end{equation}
where $\Ch([A])$ is the Chern character, 
\begin{equation}
\Ch:K^0\lrpar{T^*(\inside{M}/B)}\longrightarrow \Hc^{\even}
\lrpar{T^*(\inside{M}/B)}.
\label{rccbif.20}\end{equation}
In fact, the identification, in the three cases, of the stable homotopy
classes of the Fredholm elements with the relative K-theory involves
non-trivial homotopies. As a result \eqref{rccbif.2} is not really a
`formula' for the families index. Here we give much more explicit formul\ae\ for
the Chern character in terms of the `symbolic' data determining the
Fredholm condition for the family. In each case this corresponds to the
ellipticity of the interior symbol, uniformly up to the boundary in an
appropriate sense, together with the invertibility of a boundary family. To
represent the Chern character we construct relative deRham chain complexes,
all with cohomology $\Hc^*(T^*(\inside{M}/B)),$ tailored to each setting and
then construct Chern-Weil forms depending on the leading symbolic
data. Then $\psi_!$ is the (generalized) integration map on cohomology from
$T^*(M/B)$ to $B.$ 

To explain the strategy behind these explicit versions of the formul\ae\
\eqref{rccbif.2}, consider the familiar case of an elliptic family of
(classical) pseudodifferential operators $A\in\Psi^m(M/B;\bbE)$ where
$\bbE=(E^+,E^-)$ is a superbundle (\ie a $\bbZ_2$-graded bundle) over $M$
and the fibration \eqref{rccbif.3} has compact boundaryless fibres. This is
the families setting of Atiyah and Singer and the analytic index is a map
as in \eqref{rccbif.1}, in this case the Chern character gives
\begin{equation*}
\Ch(\Ind(A)):K(T^*(M/B))\longrightarrow H^{\even}(B).
\end{equation*}
The K-class is fixed by the (invertible) symbol of $A,$
\begin{equation*}
	a=\sigma_m (A)\in\CI(S^*(M/B);\hom(\bbE)\otimes N_m)
\end{equation*}
where $N_m$ is a
trivial real line bundle capturing the homogeneity. Fedosov in \cite{MR1401125}
gives an explicit formula for the Chern character of the compactly
supported K-class determined by $\bbE$ and $a$ as a deRham class with
compact support on $T^*(M/B).$ Modifying his approach slightly we consider
the representation of $\Hc^*(W),$ for any real vector bundle $W$ over $M,$ as
the hypercohomology of the relative complex 
\begin{equation}
\CI(M;\Lambda ^*)\oplus\CI(\bbS W;\Lambda ^*),\ D=\begin{pmatrix}d&\\ -\pi^*&-d
\end{pmatrix}
\label{rccbif.5}\end{equation}
where $\Lambda^*$ denotes the bundle of forms, $\pi:W\longrightarrow M$ is the bundle projection, and $\bbS
W=(W\setminus O_M)/\bbR^+$ is the sphere bundle of $M.$ The Chern-Weil formula
given by Fedosov in terms of connections and curvatures, naturally fits
into this representation as the class
\begin{equation}
\begin{gathered}
\Ch([A])=\Ch(\bbE,a)=\Ch(\bbE)\oplus\tCh(a)\in\CI(M;\Lambda ^{\even})\oplus
\CI(S^*(M/B);\Lambda ^{\odd})
\\
\Ch(\bbE)=\tr e^{\omega_+}-\tr e^{\omega _-},
\\
\tCh(a)=-\frac1{2\pi i}
	\int_0^1 \tr\left(a^{-1}(\nabla a)e^{w(t)}\right)dt\Mwhere
\\
w(t)=(1-t)\omega _++ta^{-1}\omega _-a + \frac 1{2\pi i} t(1-t)(a^{-1}\nabla a)^2.
\end{gathered}
\label{rccbif.6}\end{equation}
The push-forward map on cohomology, $\psi_!,$ becomes integration of the
second form in \eqref{rccbif.5} and the Atiyah-Singer formula,
\eqref{rccbif.2}, becomes an explict fibre integral
\begin{equation}
\Ch(\Ind(A))=\int_{S^*(M/B)}\left(\tCh(a)\Td(\psi)\right).
\label{rccbif.123}\end{equation}
We proceed to discuss formul\ae\ essentially as explicit as
\eqref{rccbif.6} and \eqref{rccbif.123} for the Chern character of the index
bundle for families of pseudodifferential operators on a manifold with
boundary, where the uniformity conditions up to the boundary correspond to
one of the three cases mentioned above.

The simplest of the three cases, really because it is the most commutative,
corresponds to the scattering calculus of \cite{MR95k:58168}. This arises
geometrically in the context of asymptotically flat manifolds and in
particular includes constant-coefficient differential operators on a vector
space, thought of as acting on the radial compactification as a compact
manifold with boundary. Thus, the fibration \eqref{rccbif.3} now, and from
now on, has fibres which are compact manifolds with boundary and we
consider a `fully elliptic' family $A\in\Psi_{\scat}^m(M/B;\bbE).$ As well
as a uniform version of the symbol, $a=\sigma
(A)\in\CI(\scS^*(M/B);\hom(\bbE)\otimes N_m)$ (acting on a
boundary-rescaled version of the cosphere bundle) there is an
invariantly-defined boundary symbol $b=\beta(A)\in\CI(\comp{\scT^*_{\pa
    M}(M/B)};\hom(\bbE)\otimes N_m).$ Together these two symbols form a
smooth section of the bundle over the boundary of the radial
compactification $\comp{\scT^*(M/B)}$ which is continuous at the
corner. Full ellipticity of the family $A$ reduces to invertibility of this
joint symbol and this is equivalent to the requirement that $A$ be a family
of Fredholm operators on the natural geometric Sobolev spaces. Since
$\comp{\scT^*(M/B)}$ is a topological manifold with boundary, these full
symbols provide a chain space for the K-theory (with compact supports in
the interior) and the analytic index becomes a well-defined map as in
\eqref{rccbif.1}.

To give an explicit formula for the Chern character in this scattering
setting we use a similar relative chain complex to \eqref{rccbif.5}, now
associated to a vector bundle $\pi:W\longrightarrow M$ over a manifold with
boundary. Namely, 
\begin{equation}
\begin{gathered}
\CI(M;\Lambda^*)\oplus
\left\{(u,v);u\in\CI(\bbS W;\Lambda ^*),\ v\in\CI(\comp{W}_{\pa
  M};\Lambda ^*)\Mand \iota _{\pa }^*u=\iota _{\pa}^*v\right\},\\
D=\begin{pmatrix}d&0\\ \phi^*&-d\end{pmatrix}, \
\phi = \begin{pmatrix} -\pi^* \\ -\pi^*\iota^*_{\pa} \end{pmatrix}.
\end{gathered}
\label{rccbif.7}\end{equation}
The cohomology of this chain complex is $\Hc^*(\inside{W}),$ the compactly
supported cohomology of $W$ restricted to the interior of $M,$ and the
Chern character is represented by the explicit forms 
\begin{equation}
\Ch([A])=\Ch(\bbE,a,b)=\Ch(\bbE)\oplus(\tCh(a),\tCh(b))
\label{rccbif.8}\end{equation}
where $\tCh(a)$ and $\tCh(b)$ are given by Fedosov's formula
\eqref{rccbif.6}. The Atiyah-Singer formula \eqref{rccbif.2} then follows
and in this case it becomes quite clear that the boundary symbol is a
rather complete analogue of the usual symbol and enters into the index
formula in the same way.

The other two contexts in which we give such an explicit index formula are
similar, but more complicated and less symmetric between the boundary and
the cosphere bundle (the high-momentum limit) because there is more
residual non-commutativity at the boundary. Note that we are only
considering those settings in which the analytic index gives a map as in
\eqref{rccbif.1} (always the same map of course!) This restriction
corresponds to the fact, as we shall see, that the only non-commutativity
which remains at the boundary is in the normal variables -- tangential
non-commutativity, as occurs in the cusp or b-calculi, induces an analytic
index map from a different K-theory in place of the relative K-theory in
\eqref{rccbif.1} and this leads to more subtlety in the index formula. 

Consider next the `zero calculus' which corresponds geometrically to
asymptotically hyperbolic (or `conformally compact') manifolds and is so
named because it quantizates the Lie algebra of vector fields which vanish
at the boundary of any compact manifold with boundary (as opposed to the
scattering calculus which quantizes the smaller Lie algebra in which the
normal part also vanishes to second order). As noted this Lie algebra is not
commutative at the boundary, rather it is solvable with tangential part
forming an Abelian subalgebra on which the normal part acts by
homotheity. As a result the final formula has a truly regularized Chern
character, an eta form, coming from the normal part. Again we consider a
family of pseudodifferential operators, now $A\in\Psi^m_{0}(M/B;\bbE)$ for a
fibration \eqref{rccbif.3} with fibres compact manifolds with boundary
modelled on $Z.$ The `fully ellipticity' of such a family, corresponding to
its being Fredholm on the natural `zero' Sobolev spaces, reduces to the
invertibility of the (uniform) symbol, $a=\sigma
_m(A)\in\CI(\ZS^*(M/B);\hom(\bbE)\otimes N_m)$ together with the
invertibility of the normal operator, or equivalently the reduced normal
family. The former takes values in the bundle, over the boundary, of invariant
pseudodifferential operators on the Lie group associated to the tangent
solvable Lie algebra mentioned above. The condition here is invertibility
on the appropriate Sobolev spaces since this inverse, because of the
appearance of non-trivial asymptotic terms, lies in a larger space of
pseudodifferential operators on the group. The reduced normal operator, $\RN(A),$
corresponds to decomposition of the normal operator in terms of the
representations of the solvable group.

In the case of the zero calculus it is less obvious, but shown in
\cite{Albin-Melrose1}, that this symbolic data, the invertibility of which
fixes the Fredholm property of $A,$ also gives a chain space for the
relative K-theory $[A]=[(a,\RN(A))]\in K(T^*(M/B)).$ Similarly, the chain
complex leading again to cohomology with compact supports is more involved
and depends on more of the structure of the zero cotangent bundle. More
abstractly, consider a real vector bundle $W$ over $M$ with restriction to
the boundary having a trivial real line subbundle $L\subset W_{\pa M},$ set
$U=W_{\pa M}/L$ and 
\begin{equation}
\begin{gathered}
\CI(M;\Lambda ^k)\oplus\CI(\bbS W\cup\comp{L};\Lambda^{k-1})
\oplus\CI(\bbS U;\Lambda ^{k-3}),\\
\CI(\bbS W\cup\comp{L};\Lambda^{k-1})=
\left\{(a,\gamma)\in\CI(\bbS W;\Lambda ^{k-1})\oplus\CI(\comp{L};\Lambda^{k-1});
i^*_{\pm}a=i^*_{\pm}\gamma\right\}\\
D=\begin{pmatrix}d&0&0\\ \phi_1&-d&0\\
0&\phi_2&d
\end{pmatrix},\
\phi_1=\begin{pmatrix} -\pi^*_{\bbS W} \\
-\pi^*_{\bar L} \iota_{\pa}^*\end{pmatrix},\
\phi_2=\begin{pmatrix} -\nu^{\bbS W}_* , \pi^*_{\bbS U}\nu^{\bar L}_*\end{pmatrix}.
\end{gathered}
\label{rccbif.9}\end{equation}
Here $i_{\pm}$ are the restrictions to the two boundary components of the radial compactification of $L$, 
$\comp{L}$ (which are also submanifolds of $\bbS W),$ the maps $\pi$ are
the various bundle projections and $\nu^{\bbS W}_*,$ $\nu^L_*$ are (well-defined)
push-forward maps along the fibres of $L,$ the first from $\pa\bbS W$ to
$\bbS U$ and the second from $\comp{L}$ to $\pa M.$ As already anticipated
the cohomology of this complex is canonically isomorphic to the compactly
supported cohomology of $W$ over $\inside{M}.$

The reduced normal operator for the zero calculus is a family of
pseudodiferential operators on an interval, although with different
uniformity behaviour at the two ends. It is fixed by the choice of a
splitting of the zero cotangent bundle over the boundary and the choice of
a metric and is then naturally parametrized by $S^*(\pa M/B)$ which is $\bbS
U$ in the preceeding paragraph. In terms of the representation of the
cohomology with compact supports of $T^*(\inside{M}/B)$ given by
\eqref{rccbif.9}, with $W=\ZT^*(M/B),$ the Chern character is given by the
explicit forms 
\begin{equation}
\Ch([A])=\Ch(a,\RN(A))
=\Ch(\bbE)\oplus \left( \tCh(a), \tCh(I_b(\RN(A))) \right)\oplus (-\eta(\RN(A))),
\label{rccbif.10}\end{equation}
where $I_b(\RN(A))$ is the model operator of the reduced normal operator at
one end of the interval, known as the indicial family, and $\eta(\cN)$ is
given by an expression similar to that defining $\tCh(a)$ but taking into
account that the operators involved are not trace-class and do not commute,
namely
\begin{equation*}
-\frac1{2\pi i} \int_0^1 
\lrpar{1-t}\RTr_{\fbsc}\lrpar{\cN^{-1}\lrpar{\nabla\cN}e^{\omega_{\cN}\lrpar t}}
+t\; \RTr_{\fbsc}\lrpar{\lrpar{\nabla\cN}e^{\omega_{\cN}\lrpar t}\cN^{-1}} \;dt.
\end{equation*}
The renormalized trace appearing in this formula is defined in Appendix
\ref{sec:ResTrace} following previous work of the second author and Victor
Nistor \cite{Melrose-Nistor0}. As in the scattering case we obtain an
explicit representative of the Chern character.

Finally, we consider the calculus of Boutet de Monvel, the `transmission
calculus', which contains classical elliptic boundary value problems and
their parametrices. Geometrically this calculus corresponds to Riemannian
manifolds with boundary and so the connection to relative K-theory is
immediate and already present in \cite{BoutetdeMonvel1}.  An element of the
transmission calculus is represented by a matrix of operators
\begin{equation*}
	\cA=\begin{pmatrix}
	    	\gamma^+A+B &  & K \\ \\ T & & Q 
\end{pmatrix}
:\begin{matrix}
\CI\lrpar{X;E^+} & & \CI\lrpar{X;E^-} \\ \oplus & \to & \oplus \\ 
		\CI\lrpar{\pa X;F^+} & &\CI\lrpar{\pa X;F^-} 
	 \end{matrix}
\end{equation*}
acting on the superbundles $\bbE$ over $X$ and $\bbF$ over $\pa X$.
Whether or not a family $\cA \in \Psi^m_{\tm}(M/B;\bbE, \bbF)$ is Fredholm
is again determined by invertibility of two model operators, the interior
symbol and the boundary symbol. An operator is elliptic if the former is
invertible and `fully elliptic' if they are both invertible.  Similarly to
the zero calculus, the model operator at the boundary is a family of
pseudodifferential operators parametrized by the cosphere bundle of the
boundary, but here these operators are of Wiener-Hopf type.

As before there is a convenient description of the relative cohomology of
the cotangent bundle over the interior that is compatible with the
represention of a $K$-class by a fully elliptic family of transmission
operators.  For this consider $W,$ $L,$ and $U$ as described in
\eqref{rccbif.9} above, together with a restricted space of sections of
$\bbS W,$
\begin{equation*}
\CI_{\pm}\lrpar{\bbS W;\Lambda^*} 
= \lrbrac{\alpha \in \CI\lrpar{\bbS W;\Lambda^*} : 
i^*_{L^+}\alpha = i^*_{L^-}\alpha, \; i^*_{+} d\alpha = i^*_{-} d\alpha },
\end{equation*}
and form the complex
\begin{multline}
\CI(M;\Lambda^k) 
\oplus 
\lrpar{ \CI_{\pm}(\bbS W;\Lambda^{k-1}) 
\oplus \CI(\pa M;\Lambda^{k-2}) } 
\oplus \CI(\bbS U;\Lambda^{k-3}),
\\
D=\begin{pmatrix}d&0&0\\ \phi_1&-d&0\\
0&\phi_2&d
\end{pmatrix},\
\phi_1=\begin{pmatrix} -\pi^*_{\bbS W} \\ -i^*_{\pa M}\end{pmatrix},\
\phi_2=\begin{pmatrix} \nu^{\bbS W}_*, -\pi^*_{\bbS U} \end{pmatrix}.
\label{rccbif.9'}\end{multline}
with the notation as in the previous paragraph. The cohomology of this
complex is again canonically isomorphic to the compactly supported
cohomology of $W$ over $M^\circ$.

If $a$ and $N$ are respectively the interior symbol and boundary symbol of a fully elliptic family of transmission operators, then the Chern character of the associated $K$-theory class is represented by
\begin{equation*}
\Ch(\bbE, \bbF, a, N)
= (\Ch(\bbE), \lrpar{\wt\Ch(a), -\Ch(\bbE_{\pa} \oplus \bbF)}, -\eta(N))
\end{equation*}
in the complex \eqref{rccbif.9'} with $W$, $L$, and $U$ again equal to
$T^*(M/B),$ the normal bundle to the boundary, and $T^*(\pa M/B)$ respectively.
The form $\eta(N)$ is formally the same as in the zero calculus, but the
renormalization of the trace is done in a different 
fashion, due to Fedosov (see \eqref{FedosovTrace} below).

As already mentioned, central to this discussion is the fact that the
K-theory described by fully elliptic families in these calculi is the
topological K-theory of the cotangent bundle in the interior. A well-known
consequence is that quantization of an elliptic symbol to a Fredholm
operator is only possible if the `Atiyah-Bott obstruction' of the symbol
vanish.  One way around this is to quantize into other pseudodifferential
calculi. For instance, the $b$-calculus, which is well-adapted to manifolds
with asymptotically cylindrical ends, is {\em universal} in the sense that
any elliptic symbol can be quantized to a Fredholm operator. Whereas the
calculi described above are asymptotically non-commutative in (at most) the
normal direction, the $b$ calculus is in the same sense asymptotically
non-commutative in all directions. This is related to the fact that the eta
invariants described above are `local in the boundary and global in the
normal direction to the boundary' while the eta invariant in the
Atiyah-Patodi-Singer index theorem is global in the boundary. An analysis
of the Chern character for the $b$-calculus and related calculi is the
subject of an ongoing project of the second author with Fr\'ed\'eric Rochon
\cite{Melrose-Rochon1}.

This manuscript is divided into three parts.  Section \ref{sec:Rel} is
devoted to a discussion of cohomology. We describe the approach to relative
cohomology that we will follow together with some standard properties and
work out various ways of representing $H^*_c(T^*M^\circ/B)$.  In section
\ref{sec:Kthy} the represention of a class in $\Kc(T^*M^\circ/B)$ by
a family of fully elliptic operators in either the scattering, zero, or
transmission calculus is recalled.  Finally, in section \ref{sec:Chern}, we
put these discussions together and obtain explicit formul\ae for the Chern
character
\begin{equation*}
\Ch:\Kc(T^*M^\circ/B)\longrightarrow\Hce(T^*M^\circ/B)
\end{equation*}
as described above. In each case an explicit `Atiyah-Singer' formula for
the Chern character of the index bundle is given using only the appropriate
model operators. 

\paperbody

\section{Relative cohomology} \label{sec:Rel}

Suppose that $(\cC^*_i,d),$ $i=1,\dots,N,$ are $\bbZ$-graded differential
complexes and $\phi_i:\cC_i^k\longrightarrow\cC_{i+1}^{k+1-f_i}$ for $1\le
i<N,$ is a complex of chain maps between them, so $d\phi_i=\phi_{i+1}d.$
Then the complexes can be `rolled up' into one complex. In fact only the
cases $N=2$ and $N=3$ arise here, so consider first $N=2:$
\begin{equation}
\xymatrix{
\cdots\ar[r]^d&
\cC_1^{k-1}\ar[d]^{\phi}\ar[r]^d&
\cC_1^k\ar[d]^{\phi}\ar[r]^d&
\cC_1^{k+1}\ar[d]^{\phi}\ar[r]^d&
\cdots\\
\cdots\ar[r]^d&
\cC_2^{k-f}\ar[r]^d&
\cC_2^{k+1-f}\ar[r]^d&
\cC_2^{k+2-f}\ar[r]^d&
\cdots.
} 
\label{Z-i.71}\end{equation}
Then $\phi$ induces a map of the corresponding
hypercohomologies which we can also denote
$\phi:\cH^{k}_1\longrightarrow\cH^{k+1-f}_{2}.$ The relative chain complex
\begin{equation}
(\cC^k,D),\ \cC^k=\cC^{k}_1\oplus\cC^{k-f}_2,\ D=\begin{pmatrix}d&0\\ \phi&-d
\end{pmatrix}
\label{Z-i.60}\end{equation}
is such that inclusion and projection give an exact sequence
\begin{equation}
\cC_2^{k-f}\overset{\iota}\longrightarrow\cC^k\overset{p}\longrightarrow
\cC^{k}_1.
\label{Z-i.62}\end{equation}
The hypercohomology of \eqref{Z-i.60}, denoted
$\cH^*(\cC_*,\phi),$ may be computed from a spectral sequence, in this
case a rather simple one corresponding to the fact that the long exact
sequence associated to
\eqref{Z-i.62} is
\begin{equation}
\cdots\longrightarrow \cH_2^{k-f}\overset{\iota}\longrightarrow
\cH^k(\cC_*,\phi)\overset{p}\longrightarrow \cH_1^{k}
\overset{\phi}\longrightarrow\cH_2^{k-f}\cdots.
\label{Z-i.61}\end{equation}

For $N=3$ the $\phi_i$ give a commutative diagram
\begin{equation}
\xymatrix{
\cdots\ar[r]^d&
\cC_1^{k-1}\ar[d]^{\phi_1}\ar[r]^d&
\cC_1^k\ar[d]^{\phi_1}\ar[r]^d&
\cC_1^{k+1}\ar[d]^{\phi_1}\ar[r]^d&
\cdots\\
\cdots\ar[r]^d&
\cC_2^{k-f_1}\ar[d]^{\phi_2}\ar[r]^d&
\cC_2^{k+1-f_1}\ar[d]^{\phi_2}\ar[r]^d&
\cC_2^{k+2-f_1}\ar[d]^{\phi_2}\ar[r]^d&
\cdots\\
\cdots\ar[r]^d&
\cC_3^{k+1-f_1-f_2}\ar[r]^d&
\cC_3^{k+2-f_1-f_2}\ar[r]^d&
\cC_3^{k+3-f_1-f_2}\ar[r]^d&
\cdots\\
}
\label{Z-i.70}\end{equation}
The `rolled up' double complex is 
\begin{equation}
(\cC^k,D),\ \cC^k=\bigoplus_{i=1}^3\cC^{k}_i,\
D=\begin{pmatrix}
d&0&0
\\
\phi_1&-d&0\\
0&\phi_2&d
\end{pmatrix}.
\label{Z-i.72}\end{equation}
Of course, the second two rows in \eqref{Z-i.70} are an example of
\eqref{Z-i.71}, so give the complex 
\begin{equation}
(\tcC_2,\td),\ \tcC_2^k=\cC_2^k\oplus\cC_3^{k-f_2},\
\td=
\begin{pmatrix}d&0\\ \phi_2&-d
\end{pmatrix}
\label{Z-i.73}\end{equation}
such that 
\begin{equation}
\tcC_2^{k-f_1}\overset{\iota}\longrightarrow\cC^k\overset{p}\longrightarrow
\cC^{k}_1
\label{Z-i.74}\end{equation}
is a short exact sequence. The hypercohomologies therefore give long exact
sequences as in \eqref{Z-i.61}
\begin{equation}
\begin{gathered}
\cdots\longrightarrow \tcH_2^{k-f}\overset{\iota}\longrightarrow
\cH^k(\cC_*,\phi)\overset{p}\longrightarrow \cH_1^{k}
\overset{\phi_1}\longrightarrow\tcH_2^{k+1-f}\cdots\\
\cdots\longrightarrow \cH_3^{k-f}\overset{\iota}\longrightarrow
\tcH^k\overset{p}\longrightarrow\cH_2^{k}
\overset{\phi_2}\longrightarrow\cH_3^{k+1-f}\cdots.
\end{gathered}
\label{Z-i.75}\end{equation}
Thus the case $N>2$ reduces to an iteration of $N=2$ cases.

The most obvious case of relative cohomology in a deRham setting arises from
a smooth map between compact manifolds (possibly with corners)
\begin{equation}
\psi:M\longrightarrow X,\Mwith \cC_2=\CI(M;\Lambda ^*),\
\cC_1=\CI(X,\Lambda ^*)\Mand D=\begin{pmatrix}d&0\\
\psi^*&-d
\end{pmatrix}.
\label{Map}\end{equation}
In this case we may denote the relative cohomology by $\cH^*(M,\psi).$ Note that this is precisely how relative cohomology is defined in \cite[\S 6]{Bott-Tu}.

Several variants of relative cohomology, leading to standard cohomology
groups, are of interest here, only the last (and most fundamental for the
zero index formula) corresponds to $N=3.$

\subsection{Homotopy invariance and module structure}

In all of the cases we will consider below the complexes $\cC_i^*$ will be built up from differential forms and will inherit more structure. We point out some of these properties for later use.

\subsubsection{Module structure}
First, suppose that each of the $\cC_i$ has a graded product
\begin{gather*}
	\owedge_i: \cC_i^j \times \cC_i^k \to \cC_i^{j+k} \Mst\\
d\lrpar{\alpha \owedge_i \beta} 
= d\alpha \owedge_i \beta + \lrpar{-1}^{|\alpha|} \alpha\owedge_i d\beta,\
\phi_i\lrpar{\alpha \owedge_i \beta} = \phi_i\alpha \owedge_{i+1} \phi_i\beta.
\end{gather*}

\begin{lemma}
If $N=3$ and each $\cC_i$ has a graded product as above, then
$\cH^*\lrpar{\cC_*,\phi}$ is a module over $\cH^*\lrpar{\cC_1}$ through
\begin{equation*}
\cH^*\lrpar{\cC_*,\phi}\times\cH^*\lrpar{\cC_1}\ni
\lrpar{\alpha_i,\gamma} \overset\owedge \longmapsto
		\lrpar{\alpha_1\owedge_1\gamma, \; 
		\alpha_2\owedge_2\phi_1\gamma, \;
		\alpha_3\owedge_3\phi_2\phi_1\gamma)\in\cH^*\lrpar{\cC_*,\phi}}
\end{equation*}
\end{lemma}

\begin{proof}
Notice that if $\lrpar{\alpha_i} \in \cC_*^k$ 
\begin{multline*}
	D\lrpar{\lrpar{\alpha_i}\owedge \gamma}
	= \begin{pmatrix}
	d\lrpar{\alpha_1\owedge_1\gamma} \\ 
	\phi_1 \lrpar{\alpha_1\owedge_1\gamma}
		- d\lrpar{\alpha_2\owedge_2\phi_1\gamma}\\
	\phi_2\lrpar{\alpha_2\owedge_2\phi_1\gamma}
		+d\lrpar{\alpha_3\owedge_3\phi_2\phi_1\gamma}
	  \end{pmatrix} \\
	= D\lrpar{\alpha_i} \owedge \gamma
	+\lrpar{-1}^k \begin{pmatrix}
	 \alpha_1  \\ 
		- \lrpar{-1}^{-f_1} \alpha_2\\
	\lrpar{-1}^{-f_1-f_2} \alpha_3 
	  \end{pmatrix} \owedge d\gamma.
\end{multline*}
Thus if $f_1$ and $f_2$ are odd $D$ satisfies a Leibnitz rule with respect
to $\owedge.$ In any case it is true that if $\gamma$ is closed and
$\lrpar{\alpha_i}$ is $D$-closed then $\lrpar{\alpha_i}\owedge\gamma$ is
$D$-closed and represents a class in $\cH^*\lrpar{M,\phi}$ depending only
on the class of $\gamma$ in $\cH^*\lrpar{\cC_1}$ and $\lrpar{\alpha_i}$ in
$\cH^*\lrpar{M,\phi}.$
\end{proof}

\subsubsection{Homotopy invariance}\label{sec:HtpyInv} Recall the usual
proof of homotopy invariance of the Chern character on a closed manifold.
A one-parameter family of connections on a fixed bundle over a space
$X,$ can be interpreted as a single connection on the same bundle
pulled-back to $X\times\lrspar{0,1}.$ Since the Chern character of this connection
is closed, the cohomology class of the Chern character of the family of
connections is constant in the parameter. What is important here is that the
space of forms on $X\times\lrspar{0,1}_r$ is really two copies of the space
of forms on $X,$
\begin{equation*}
	\Omega^*\lrpar{X\times\lrspar{0,1}_r}\
	\cong \Omega^*X \oplus dr\wedge \Omega^{*-1}X
\end{equation*}
and that the differential becomes $\begin{pmatrix} d & \\ \pa_r & -d
\end{pmatrix}$ with respect to this splitting.

The spaces $\cC_i$ below will generally be direct sums of spaces of
differential forms with perhaps some compatibility conditions. In order to
carry out the classical argument for the homotopy invariance of the Chern
character, assume that a one-parameter family of elements
$\lrpar{\alpha\lrpar r}$ of $\cC_*$ define an element of $\wt\cC_*$ with
\begin{equation}
	\wt\cC^i_* 
	= \cC^i_* \oplus \cC^i_{*-1} 
	\quad \lrpar{= \cC^i_* + dr\wedge\cC^i_{*-1} }
\label{NormalSplitting}\end{equation}
and that the action of $d$ and $\phi_i$ is extended to $\wt\cC_*$ so that
with respect to this splitting 
\begin{equation*}
	d = \begin{pmatrix}  d & \\ \pa_r & -d \end{pmatrix}
	\Mand \phi_i = \begin{pmatrix} \phi_i & \\ & \phi_i \end{pmatrix}.
\end{equation*}
Then $D$ acts on $\wt\cC_* = \cC_* \oplus \cC_{*-1}$ by
\begin{equation*}
\wt D = 
\begin{pmatrix} d & & & & & \\ \phi_1 & -d & & & & \\ & \phi_2 & d & & & \\
    \pa_r & & &-d & & \\ & -\pa_r & & \phi_1 & d & \\ & & \pa_r & & \phi_2 & -d
\end{pmatrix}
\end{equation*}

\begin{lemma}\label{HomotopyInv}
If the spaces $\cC^i_*$ have the properties of differential forms as
described above and $\lrpar{\wt\alpha\lrpar r}$ is a $\wt D$ closed element
of $\wt\cC_*$ then $\wt\alpha\lrpar 0$ and $\wt\alpha\lrpar 1$ are
cohomologous in $\cC_*.$ 
\end{lemma}

\begin{proof}Write
\begin{equation*}
	\wt\alpha^i 
	= \lrpar{\alpha^i_t, \alpha^i_n}
	\quad \lrpar{ = \alpha^i_t + dr\wedge \alpha^i_n }
\end{equation*}
with respect to the splitting \eqref{NormalSplitting}. Then the fact that
$\lrpar{\wt\alpha}$ is $\wt D$-closed implies that
\begin{multline*}
	i^*_0\wt\alpha - i^*_1\wt\alpha
	= \int_0^1 \pa_r\lrpar{\alpha_t} \;dr 
	= \int_0^1 \lrpar{ d\alpha^1_n, \phi_1\alpha^1_n - d\alpha^2_n,
		 -\phi_2\alpha^2_n + d\alpha^3_n} \;dr \\
	= D\lrpar{ \int_0^1 \lrpar{\alpha^1_n, -\alpha^2_n, \alpha^3_n} \;dr},
\end{multline*}
hence $	i^*_0\wt\alpha$ and $i^*_1\wt\alpha$ represent the same class in
the cohomology of $\lrpar{\cC_*,D}.$ 
\end{proof}

\subsubsection{Push-forward}
Notice that, if $M$ is oriented and $\dim Y < \dim M,$ there is a
well-defined integral
\begin{equation}
\int_M:\cH(M,\psi)\longrightarrow \bbC.
\label{Z-i.11}\end{equation}
Indeed, two representatives of the same class in $\cH\lrpar{M,\psi}$ differ
an element in the image of $D,$ i.e. of the form $\lrpar{u,\psi^*u-dv}$,
but $\int_M dv=0$ by Stoke's theorem and $\int_M \psi^*u=0$ for dimensional
reasons, so the value of the integral is not affected.

More generally, if we have a pair $(\cC^*_1\oplus \cC^*_2)$,
$(\wt\cC^*_1\oplus \wt\cC^*_2)$ of $\bbZ$-graded complexes with chain
maps $\phi$, $\wt\phi$ as in \eqref{Z-i.60}, and maps between them
\begin{equation*}
\cC_i^* \xrightarrow{\psi_i} \wt\cC_i^{*-\ell}
\end{equation*}
satisfying $d\psi_i = \psi_i d$ and fitting into the commutative diagram
\begin{equation}
\xymatrix{ \cC_1^k \ar[r]^-{\phi} \ar[d]^{\psi_1} & \cC_2^{k+1-f}
\ar[d]^{\psi_2} \\
\wt\cC_1^{k-\ell} \ar[r]^-{\wt\phi} & \wt\cC_2^{k-\ell+1-f} }
\label{PreCohoMap}\end{equation}
we get a map 
\begin{equation}
\cH^*\lrpar{\cC_*,\phi} \xrightarrow{\psi} \cH^{*-\ell}\lrpar{\wt\cC_*,\wt\phi}.
\label{InducedCohoMap}\end{equation}

This situation occurs for instance for pull-back diagrams: Assume
$X\xrightarrow{f} Z$ is a fibration of oriented manifolds and
$E\xrightarrow{\pi} Z$ is a vector bundle, then we have
\begin{equation*}
\xymatrix{ f^*\bbS E \ar[r]^{\wt\pi} \ar[d]^{\wt f} & X \ar[d]^f \\
\bbS E \ar[r]^\pi & Z }
\phantom{xx}
\underset{\underset{\underset{\Longrightarrow}{}}{}}{}
\phantom{xx}
\xymatrix{ \Omega^*\lrpar{f^*\bbS E}  \ar[d]^{\wt f_*} &
\Omega^*X \ar[d]^{f_*} \ar[l]_-{\wt\pi^*} \\
\Omega^{*-\ell}\lrpar{\bbS E}  & \Omega^{*-\ell}Z \ar[l]_-{\pi^*}}
\end{equation*}
where $\ell=\dim\lrpar{X/Z}$ and $f_*$ denotes the push-forward of
differential forms by `integration along the fibers'. Then
\eqref{PreCohoMap} commutes (essentially because the fibers of $f$ and $\wt
f$ coincide) and we get a map $\cH^*\lrpar{f^*\bbS E, \wt\pi} \to
\cH^*\lrpar{\bbS E, \pi}$ induced by $f$ via \eqref{InducedCohoMap}, which
we can call push-forward by $f$.

Alternately, consider a smooth fibration $N \xrightarrow{h} \Gamma$ with
$\dim N >\dim \Gamma$ and the vertical cosphere bundle
$\bbS^*\lrpar{N/\Gamma} \xrightarrow{\pi} N$.  We can use
\eqref{InducedCohoMap} to get a map
\begin{equation}
\cH^*\lrpar{\bbS^*\lrpar{N/\Gamma},\pi} \xrightarrow{h_*} H^*\lrpar Y
\label{FibPush}\end{equation}
by taking $\ell=2\dim \lrpar{N/\Gamma}-1$, $\wt\cC_1^* = 0$,
$\wt\cC_2^*=\Omega^*Y$ and mapping between $\Omega^*\bbS^*\lrpar{N/\Gamma}$
and $\Omega^*Y$ by the push-forward of forms along the map $h\circ\pi$. In
this case \eqref{PreCohoMap} commutes because forms pulled back from $N$
push-foward to zero along $h\circ\pi$. This push-forward map will be used
in the formula for the families index theorem below. (Note that
\eqref{Z-i.11} is a particular case with $\Gamma=\{pt\}$).

\subsection{Manifold with boundary}\label{rel-mwb}
A standard case of \eqref{Map} arises when $X$ is a compact manifold with
boundary and $\psi$ is the inclusion map for the boundary. Thus
$\cC_1=\CI(X;\Lambda ^*)$ and $\psi=\iota_{\pa}:\pa X\hookrightarrow X,$
$\cC_2=\CI(\pa X;\Lambda^*).$ 
Then

\begin{lemma}\label{rccbif.12} For $\iota:\pa X\longrightarrow X$ the
 inclusion of the boundary of a compact manifold with boundary and
\begin{equation}
\begin{gathered}
\cC_*^k=\CI(X;\Lambda^k)\oplus\CI(\pa X;\Lambda^{k-1}),\\
\cH^k(\cC_*,\phi)=\cH^k(\pa X,\iota)=H^k(X;\pa X)=\Hc^k(X\setminus\pa X)
\end{gathered}
\label{Z-i.86}\end{equation}
is relative cohomology in the usual sense.
\end{lemma}

\begin{proof} This is completely standard but a brief proof is included for the
sake of completeness and for later generalization. 

If $\lrpar{u,v}\in \cC_k$ satisfies $D\lrpar{u,v}=0$ then there is
$T\lrpar{u,v}\in \cC_{k-1}$ such that $DT\lrpar{u,v}-\lrpar{u,v} =
\lrpar{\omega,0}$ with
\begin{equation}
\omega \in \CIc(\inside X;\Lambda ^k).
\label{Z-i.34}\end{equation}

Taking a product neighbourhood of the boundary, $U=[0,1)_x\times\pa X,$ and
  denoting $Y=\pa X$ a smooth form decomposes on $U$ as
\begin{equation}
u=u_{t}(x)+dx\wedge u_n(x),\ u_t\in\CI([0,1)\times Y;\Lambda ^kY),\
  u_n\in\CI([0,1)\times Y;\Lambda ^{k-1}Y) 
\label{Z-i.30}\end{equation}
with differential 
\begin{equation}
du=d_Yu_t+dx\wedge \lrpar{\frac{\pa}{\pa x}u_t(x)-d_Yu_n(x)}.
\label{Z-i.31}\end{equation}
So if $(u,v)$ satisfies $D(u,v)=0,$
\begin{equation}
d_Yu_t=0,\ \frac{\pa}{\pa x}u_t(x)-d_Yu_n(x)=0\text{ in } x<1,\ dv=u_t(0).
\label{Z-i.32}\end{equation}
Choosing a cutoff function $\rho \in\CI([0,1])$ with $\rho (x)=1$ in
$x<\frac12,$ $\rho (x)=0$ in $x>\frac34,$ consider 
\begin{equation}
	T(u,v)=\lrpar{ \rho(x) \int_0^x u_n(s) \;ds + \rho(x)v , 0 }.
\label{Z-i.33}\end{equation}
This satisfies $DT(u,v)=(u+\omega,v)$ where $\omega\in\CI(X;\Lambda ^k)$ has support away
from the boundary as required in \eqref{Z-i.34}.
Thus the complex retracts to the subcomplex of deRham forms with compact
support in the interior and the cohomology is therefore the cohomology of
$X$ relative to its boundary.
\end{proof}

\begin{corollary}\label{rccbif.13} 
\label{Sphere-bundle} A particular case of \S\ref{Map} arises with
$X=\comp{U},$ the radial compactification of a real vector bundle $U$ over
a compact manifold $Y$ and $\pa X=\bbS U=(U\setminus
0_Y)/\bbR^+$ the sphere bundle, so
\begin{equation}
\cH^*(\pa X,\iota)=\Hc^*(U).
\label{rccbif.14}\end{equation}
\end{corollary}

\subsection{Sphere bundle of a real vector bundle} \label{SphereBundle}

Although Corollary~\ref{rccbif.13} is a `compact' representation of the
compactly-supported cohomology of a real vector bundle over a manifold it
is not the most natural one for index theory. In view of the contractibility of
the fibres, instead of the inclusion of the sphere bundle as the boundary of
the radial compactification of $W$ we may consider instead simply the
projection
\begin{equation}
\pi:\bbS W\longrightarrow X.
\label{rccbif.15}\end{equation}

Denote by $\cH^k\lrpar{\bbS W,\pi}$ the cohomology of the complex 
\begin{equation}
	\CI(X;\Lambda ^*)\oplus\CI(\bbS W;\Lambda ^{*-1}),\quad
	D=\begin{pmatrix} d & 0 \\ -\pi^* & -d \end{pmatrix}
\label{rccbif.17}\end{equation}
(We get the same cohomology with $\pi^*$ instead of $-\pi^*$, but the latter leads to better signs in the expressions for the Chern characters below.)

\begin{lemma}\label{rccbif.16} For any real vector bundle over a compact
  manifold without boundary,
\begin{equation*}
	\cH^k\lrpar{\bbS W,\pi} \cong H^k_c\lrpar W.
\end{equation*}
\end{lemma}

\begin{proof}

Recall that the cohomology with compact supports of $W$ may be represented
by the deRham cohomology of smooth forms with compact support on $W.$ Let
$i_0:X\hookrightarrow W$ be the inclusion of the zero section and choose a
metric on $W$ so as to have a product decomposition
\begin{equation}
W\setminus i_0\lrpar X=\bbS W\times\bbR^+
\label{Z-i.3}\end{equation}
and denote the projection onto the left factor by $R$.

Given a closed $k$-form $u$ on $W$ with compact support consider the map
\begin{equation}
\Phi\lrpar u =
\lrpar{ i_0^*u, (-1)^{k-1} R_*u } \in \CI(X;\Lambda ^k)\oplus\CI(\bbS W;\Lambda ^{k-1}).
\label{MapPhi}\end{equation}
We can think of introducing polar coordinates as pulling-back to the space $\bbS W \times \bbR^+$, say via a map $\beta$, in terms of which we have
\begin{gather*}
\beta^*u = u_t + u_n \wedge dr,
\Mwith \df i_{\pa_r} u_t =0 \\
\Mso \Phi\lrpar u = \lrpar{ u_t\lrpar 0, (-1)^{k-1} \int_0^\infty u_n \;dr}
\end{gather*}
where we denote the interior product with $\pa_r$ by $\df i_{\pa_r}.$ Notice that
\begin{equation*}
\begin{split}
\Phi\lrpar{dv} 
&= \lrpar{d_X i_0^*v,
(-1)^k R_*\lrspar{ (d_{\bbS W}v_n +(-1)^k \pa_r v_t) \wedge dr} } \\
&= \lrpar{d_X i_0^*v, -\pi^*i_0^*v - d_{\bbS W} [(-1)^{k-1} R_*v]}
= D\Phi\lrpar v,
\end{split}
\end{equation*}
so $\Phi$ defines a map on cohomology which is easily seen to be an
isomorphism (for instance by using the commutative diagram
\begin{equation*}
\xymatrix{
	\ldots \ar[r] &
	H^{k-1}\lrpar{\bbS W} \ar[r] &
	\cH^k\lrpar{\bbS W,\pi} \ar[r] &
	H^k\lrpar X \ar[r]^{\pi^*} &
	H^k\lrpar{\bbS W} \ar[r] & \ldots \\
	\ldots \ar[r] &
	H^{k-1}\lrpar{\bbS W} \ar[r] \ar[u]^{\mathrm{id}} &
	H^k_c\lrpar W \ar[r] \ar[u]^\Phi&
	H^k\lrpar{\bar W} \ar[r]^{\pi^*} \ar[u]^{i^*_0} &
	H^k\lrpar{\bbS W} \ar[r] \ar[u]^{\mathrm{id}} & \ldots }
\end{equation*}
and the Five Lemma).
\end{proof}

\subsection{Bundle with line subbundle}\label{rel-bound-zero} A variant of
the setting of Lemma~\ref{rccbif.16} arises in the index formula for
perturbations of the identity in the zero algebra. There a real vector
bundle, $W\longrightarrow Y,$ has a trivial line subbundle $L\subset W.$
Setting $U=W/L$
\begin{equation}
\cH^*(\bbS U,\pi)=\Hc^{*}(U)=\Hc^{*+1}(W)
\label{Z-i.87}\end{equation}
since $W\simeq U\oplus L$ with $L$ by assumption trivial.
For our purposes there is another more useful complex giving the same cohomology.

Consider the radial compactification of $L$, $\bar L$, obtained by attaching to each fiber of $L$ the points at $\pm\infty$, $L^+$, $L^-$, and the forms 
\begin{equation*}
	\CI_{\pm}(\bar L;\Lambda ^*)
	= \lrbrac{ \alpha \in \CI(\bar L;\Lambda ^*) : 
	i^*_{L^+} \alpha = i^*_{L^-}\alpha, \;
	i^*_{L^+} d\alpha = i^*_{L^-}d\alpha }.
\end{equation*}
We use the deRham complexes 
$\cC_1^k=\CI_{\pm}(\bar L;\Lambda ^k)$ and 
$\cC_2^k=\CI(\bbS U;\Lambda ^k)$
with the chain map
\begin{equation*}
-\pi^*\nu^L_*:\cC_1^k\longrightarrow\cC_2^{k-1},
\end{equation*}
(where $\nu^L_*:\CI_{\pm}(\bar L;\Lambda ^*)\longrightarrow \CI(Y;\Lambda ^{*-1})$ is
push-forward under $\pi_{\bbS W}$ restricted to $L$) to form the complex
$\cC^k = \cC_1^k \oplus \cC_2^{k-2}$  (note $f=2$) with differential 
$\begin{pmatrix}d & 0 \\ -\pi^*\nu^L_* & -d \end{pmatrix}$. The cohomology of this complex will be denoted 
$\cH^*\lrpar{\bbS U, \bar L}$.
Notice that for $\alpha \in \CI_{\pm}(\bar L;\Lambda ^*)$ we have $d\nu^L_*\alpha = \nu^L_*d\alpha$ (cf. Lemma \ref{PrepZero} below).

\begin{lemma}\label{rccbif.18} There are natural isomorphisms in cohomology
\begin{equation}
\cH^k(\cC,-\pi^*\nu^L_*)=\Hc^k(W)= \cH^{k-1}(\bbS U,\pi).
\label{Z-i.85}\end{equation}
\end{lemma}

\begin{proof}
We have the following commutative diagram relating the long exact sequences for
$\cH^*\lrpar{\bbS U, \bar L}$ and $\cH^*\lrpar{\bbS U,\pi}$
described in \eqref{Z-i.61},
\begin{equation*}
	\xymatrix @C = 21pt { 
	\ldots \ar[r] & H^{k-2}\lrpar{\bbS U} \ar[r] \ar[d]^{\mathrm{id}} &
	\cH^k\lrpar{\bbS U, \bar L} \ar[r] \ar[d]^{\nu^{\bar L}_*}&
	H^k\lrpar{\bar L} \ar[r]^{-\pi^*\nu^L_*} \ar[d]^{\nu^{\bar L}_*} &
	H^{k-1}\lrpar{\bbS U} \ar[r] \ar[d]^{\mathrm{id}} & \ldots \\
	\ldots \ar[r] & H^{k-2}\lrpar{\bbS U} \ar[r] &
	\cH^{k-1}\lrpar{\bbS U, \pi} \ar[r] &
	H^{k-1}\lrpar{Y} \ar[r]^{-\pi^*} &
	H^{k-1}\lrpar{\bbS U} \ar[r] & \ldots }
\end{equation*}
Since the cohomologies of $\bar L$ and $Y$ are isomorphic via $\nu^{\bar
L}_*$ the Five Lemma shows that  
\begin{equation*}
	\cH^k\lrpar{\bbS U, \bar L}
	\xrightarrow[\cong]{\nu^{\bar L}_*}
	\cH^{k-1}\lrpar{\bbS U, \pi}
	\cong H^{k-1}_c\lrpar U
	\cong H^k_c\lrpar W
\end{equation*}
as required.
\end{proof}

For future reference we point out that from the proof of this lemma and that of Lemma~\ref{rccbif.16} the map
\begin{equation}
	\Phi_{\bar L}: \CI_c\lrpar{W;\Lambda^k}
	\to \CI_{\pm}\lrpar{\bar L;\Lambda^k} \oplus \CI\lrpar{\bbS U;\Lambda^{k-2}}
\label{MapPhiL}\end{equation}
defined 
by $\Phi_{\bar L}(\omega)=\lrpar{i^*_{\bar L}\omega, (-1)^{k-1} \nu^{\bar L}_*R^W_*\omega}$, where $i_L$ is the inclusion of $L$ into $W$ and $\nu^{\bar L}_*$ is the push-forward from $\bbS W$ to $\bbS U$, induces an isomorphism in cohomology between $H^k_c(W)$ and $\cH^{k-1}\lrpar{\bbS U, \bar L}$.

\subsection{Bundle over manifold with boundary}\label{rel-scat} A more
general case, which arises in the index 
formula for the scattering algebra, corresponds to a vector bundle $W$ over
a compact manifold with boundary $X.$ Let $\comp{W}$ be the radial
compactification of $W.$ Its boundary consists of two hypersurfaces, $\bbS
W,$ the part `at infinity' and $\comp{W}_{\pa X},$ the part over the
boundary. Set
\begin{multline}
\cC_1=\CI(X;\Lambda^*)\Mand\\
\cC_2=\left\{(\alpha,\beta);\alpha\in\CI(\bbS W;\Lambda ^*),\ \beta\in\CI(\comp{W}_{\pa
  X};\Lambda ^*)\Mand \iota _{\pa}^*\alpha=\iota _{\pa}^*\beta\right\},
\label{Z-i.66}\end{multline}
the pairs of forms on the two hypersurfaces with common restriction
to the corner $\bbS W_{\pa X}.$
This gives a complex as in \eqref{Z-i.62} with $\cC^k= \cC_1^k\oplus
\cC_2^{k-1}$ and differential
\begin{equation}
	D = \begin{pmatrix}
		d & 0 \\ \phi & -d  
	\end{pmatrix},
	\Mwhere
	\phi = \begin{pmatrix} -\pi^* \\ -\pi^*i^* \end{pmatrix}
\label{Z-i.68}\end{equation}

\begin{lemma}\label{rccbif.19} The cohomology of the complex \eqref{Z-i.66}
  with differential \eqref{Z-i.68} reduces to the cohomology with compact
  support in $W_{X\setminus\pa X}.$
\end{lemma}

\begin{proof}
The cohomology of $\bar W_{\pa X}$ is canonically isomorphic to the cohomology of $\pa X$ under the pull-back map. Thus a closed form on $\bar W_{\pa X}$ is the sum of a form on $\pa X$ pulled-back to $\bar W_{\pa X}$ and an exact form on $\bar W_{\pa X}$. We point out that the same is true for a form $\beta$ on $\bar W_{\pa X}$ if $d\beta$ is a form pulled-back from $\pa X$ (this follows from the Hodge decomposition of forms and the previous statement or from the proof of \cite[Cor. 4.1.2.2]{Bott-Tu}).

Thus from $D\lrpar{a,\lrpar{\alpha,\beta}}=0$ it follows that
$\beta=\pi^*\beta' + d\beta''$ and hence
\begin{equation}
	da=0, \quad -\pi^*a = d\alpha, \quad -i^*a= d\beta'.
\label{Z-i.90}\end{equation}
Choose a product neighborhood of the boundary $\cU \cong
\lrspar{0,1}_x\times \pa X$ and a corresponding decomposition of $\bbS
W\rest{\cU} \cong  \lrspar{0,1}_x\times V$ with $V=\bbS W\rest{\pa X}.$ In
this neighborhood
\begin{align}
a = a_t\lrpar x +  a_n\lrpar x \wedge dx, \quad
a_t \in \CI(\pa X; \Lambda^k), \; a_n \in \CI(\pa X;\Lambda^{k-1}),
\label{ProdDecomp1}
\\ 
\alpha = \alpha_t\lrpar x +  \alpha_n\lrpar x \wedge dx, \quad
\alpha_t \in \CI(V; \Lambda^{k-1}), \; \alpha_n \in \CI(V;\Lambda^{k-2}),
\label{ProdDecomp2}\end{align}
and \eqref{Z-i.90} becomes
\begin{gather*}
	d_{\pa X}a_t=0, \quad 
	(-1)^k \pa_x a_t + d_{\pa X}a_n = 0, \quad
	d_V\alpha_t = -\pi^*a_t, \\
	(-1)^{k}\pa_x \alpha_t + d_V\alpha_n = -\pi^*a_n, \quad
	-a_t\lrpar 0 = d\beta'.
\end{gather*}

Choose a smooth function $\rho\in\CI\lrpar X$ that is identically equal to
one if $x<1/2$ and identically equal to zero if $x>3/4$ and define $T \in
\cC^{k-1}$ by
\begin{equation*}
T= \rho\lrpar x
\lrpar{ (-1)^{k+1}\int_0^x a_n\lrpar s \; ds - \beta', 
\lrpar{(-1)^{k-1}\int_0^x \alpha_n\lrpar s \; ds + \iota_{\pa}^*\beta'',-\beta''}}.
\end{equation*}
Then, for some $\omega\in\cC^k,$
\begin{equation*}
	DT 
	= \rho\lrpar x\lrpar{ a,\lrpar{\alpha,\beta}} + \rho'\lrpar x \omega
\end{equation*}
so that $\lrpar{a,\lrpar{\alpha,\beta}}- DT=\lrpar{\wt a, \lrpar{\wt
    \alpha,0}}$ with $\wt a$ and $\wt \alpha$ forms supported in
$X^\circ.$ Thus the complex retracts to the subcomplex of forms supported
in $X^\circ,$ and from Lemma \ref{rccbif.16} this complex computes the
cohomology with compact support of $W$ over $X^\circ.$
\end{proof}

\subsection{Bundle with line subbundle over the boundary}\label{Zero-ind-coh} 

The representations of relative cohomology that will be used for the index
of zero operators and for operators in Boutet de Monvel's transmission
calculus are closely related. Consider a compact manifold with boundary,
$X,$ a real vector bundle $W$ over $X$ which over the boundary has a
trivial line subbundle $L,$ and denote the quotient bundle over the
boundary by $U=W_{\pa X}/L.$ The compactly supported cohomology of $W$ will
be represented as in \S\ref{SphereBundle}, that of its restriction to the
boundary either as in \S\ref{rel-bound-zero} for the zero calculus or
\S\ref{SphereBundle} for the transmission calculus. An appropriate version
of the inclusion map then gives the compactly supported cohomology of the
interior much as in \S\ref{rel-mwb}. 

Notice that $\bbS W_{\pa X}\setminus\lrbrac{L^+,L^-}$ fibers over $\bbS U$
and can be identified with $\bbS U \times \bbR$ (since $L$ is trivial),
thus there is a push-forward map $\nu^{\bbS W}_*: \CI\lrpar{\bbS
  W,\Lambda^k} \to \CI\lrpar{\bbS U,\Lambda^{k-1}}$ which however does not
commute with $d.$ 

\begin{lemma}\label{PrepZero} If $\alpha \in \CI\lrpar{\bbS W,\Lambda^k}$
then
\begin{equation}\label{PushFwdLemmaZ}
\nu^{\bbS W}_*d_{\bbS W}\alpha = d_{\bbS U}\nu^{\bbS W}_*\alpha 
+ (-1)^k ( \pi^*i_{L^+}\alpha-\pi^*i_{L^-}\alpha ).
\end{equation}
Thus if $i_{L^+}\alpha=i_{L^-}\alpha$ then $\nu_L^*d\alpha = d\nu_L^*\alpha.$
\end{lemma}

\begin{proof}
Introduce polar coordinates around $L^{\pm}$ in $\bbS W_{\pa X}$ (i.e.,
blow them up) to get a map
\begin{equation}
\bbS U \times \bar L \xrightarrow{\beta} \bbS W_{\pa X}.
\label{rccbif.21}\end{equation}
The pre-image of $L^{\pm}$ will still be denoted $L^{\pm}.$ The
push-forward is given by $a \mapsto \nu^{\bar L}_*\beta^*a$ and there are
no integrability issues since $\bbS W_{\pa X}$ is compact.

In local coordinates, for $a$ a form of degree $k,$
\begin{multline*}
\beta^*a = a'\lrpar s + a''\lrpar s \wedge ds,
\Mwith a', a'' \in \CI\lrpar{\bbS U, \Lambda^*} \\
\implies	d\beta^*a = d_{\bbS U}a'\lrpar s 
+ \lrpar{ d_{\bbS U} a''\lrpar s + (-1)^k \pa_s a'\lrpar s } \wedge ds.
\end{multline*}
Hence $\nu^{\bar L}_*\beta^*a = \int_{\bar \bbR} a''\lrpar s\; ds$ and
\begin{equation*}
\nu^{\bar L}_*\beta^*\lrpar{d_{\bbS W}a} 
= d_{\bbS U} \nu^L_*\beta^*a + (-1)^k \int_{\bar \bbR} \pa_s a'\lrpar s \; ds
\end{equation*}
giving \eqref{PushFwdLemmaZ}.
\end{proof}
 
Introducing the complexes and chain maps
\begin{equation*}
\begin{gathered}
\cZ^k_1=\CI(X;\Lambda^k),\ \cZ^k_3=\CI(\bbS U;\Lambda ^k)\\
\cZ^k_2=\left\{(\alpha,\gamma)\in\CI(\bbS W;\Lambda ^k)
\oplus\CI(\comp{L};\Lambda^k);
i^*_{\pm}\alpha=i^*_{\pm}\gamma\right\}\\
\phi_1=\begin{pmatrix} -\pi^*_{\bbS W} \\
-\pi^*_{\bar L}\iota_{\pa}^*\end{pmatrix},\
\phi_2=
\begin{pmatrix} -\nu^{\bbS W}_* , \pi^*_{\bbS U}\nu^{\bar L}_*\end{pmatrix},
\end{gathered}
\end{equation*}
we define the total chain space $\cZ^*$ by
\begin{equation}
\cZ^k = \cZ^k_1 \oplus \cZ^{k-1}_2 \oplus \cZ^{k-3}_3,\quad
D_{\cZ} = \begin{pmatrix} d & & \\ \phi_1 & -d & \\ & \phi_2 & d \end{pmatrix}
\label{ZeroComplex}\end{equation}
where $i_{\pm}$ are the inclusion (or attaching) maps for the two boundary
manifolds (each canonically diffeomorphic to $\pa X)$ $L_{\pm}$ of
$\comp{L},$ either to $\bbS W$ or $\comp{L},$ and $\pi_{\bbS W}:\bbS
W\longrightarrow X,$ $\pi_{\bbS U}:\bbS U\longrightarrow \pa X$, and
$\pi_{\bar L}:\comp{L}\longrightarrow \pa X$ denote the various bundle
projections.

As in Lemma \ref{PrepZero}, neither the push-forward for $\alpha$ nor
$\gamma$ commute with $d,$ however the consistency condition
$i^*_{\pm}\alpha=i^*_{\pm}\gamma$ yields $d\phi_2 = \phi_2 d;$ together
with $d\phi_1=\phi_1d$ this ensures that $D^2_{\cZ}=0.$

With the same notation set
\begin{equation*}
\begin{gathered}
\CI_{\pm}\lrpar{\bbS W;\Lambda^*} 
= \lrbrac{\alpha \in \CI\lrpar{\bbS W;\Lambda^*} : 
i^*_{L^+}\alpha = i^*_{L^-}\alpha, \; i^*_{L^+} d\alpha = i^*_{L^-} d\alpha } \\
\cT^k_1 =\CI(X;\Lambda^k), \quad 
\cT^k_2 = \CI_{\pm}(\bbS W;\Lambda^k) \oplus \CI(\pa X; \Lambda^{k-1}), \quad
\cT^k_3 = \CI(\bbS U;\Lambda^k) \\
\Phi_1 = \begin{pmatrix} -\pi^*_{\bbS W} \\ i^*_{\pa X} \end{pmatrix},
\Phi_2 = \begin{pmatrix} \nu^{\bbS W}_*, \pi^*_{\bbS U} \end{pmatrix}
\end{gathered}
\end{equation*}
and define the total chain space $\cT^*$ by
\begin{equation}
\cT^k = \cT^k_1 \oplus \cT^{k-1}_2 \oplus \cT^{k-3}_2 ,\quad
D_{\cT} = \begin{pmatrix} d & & \\ \Phi_1 & -d & \\ & \Phi_2 & d \end{pmatrix}.
\label{TransComplex}\end{equation}
Note that $D_{\cT}^2=0$ because $i^*_{L^+}\alpha = i^*_{L^-}\alpha$
guarantees that $d\nu^L_* = \nu^L_*d.$ 

The point of these rather involved constructions of the cohomology with
compact supports is that the Chern character derived from the symbolic
data (symbol and normal operator) of a fully elliptic zero operator has a natural
representative in the chain space $\cZ_*,$ while the Chern character
constructed from the symbols (interior and boundary) of a fully elliptic
operator in the transmission calculus has a natural representative in the chain space
$\cT_*.$

\begin{lemma} The cohomology of the complexes \eqref{ZeroComplex} and
  \eqref{TransComplex} are isomorphic to the compactly supported cohomology
  of $W$ restricted to the interior of $X,$
\begin{equation}
\cH^k(\cT_*;D_T) = \cH^k(\cZ_*;D_Z) = \Hc^k(W\big|_{X\setminus\pa X}).
\label{Z-i.88}\end{equation}
\end{lemma}

\begin{proof}
The map
\begin{equation*}
	\CI_{\pm}\lrpar{\bar L;\Lambda^k} 
	\oplus
	\CI\lrpar{\bbS U; \Lambda^{k-2}}
\ni	\lrpar{\gamma, \beta}\longmapsto 
\lrpar{0,0,\gamma- \pi^*_{\bar L} i^*_{L^+}\gamma, \beta}\in
\cC_k\lrpar{\cZ_*}
\end{equation*}
fits into the short exact sequence of complexes
\begin{equation*}
	0 \to 
	\cC_*\lrpar{\bbS U, \bar L} \xrightarrow{}
	\cZ_* \xrightarrow{} 
	\cC_*\lrpar{\bbS W, \pi}
	\to 0.
\end{equation*}
This in turn induces the long exact sequence in cohomology in the top row of
\begin{equation*}
\xymatrix @C=16pt
{ \ldots \to \cH^k\lrpar{\bbS W, \pi} \ar[r]^-J & 
	\cH^k\lrpar{\bbS U,\bar L} \ar[r] &
	\cH^{k+1}\lrpar{\cZ_*;D_Z} \ar[r] & 
	\cH^{k+1}\lrpar{\bbS W,\pi} \to \ldots \\
	\ldots \to H^k_c\lrpar{W} \ar[r]^-{i_*} \ar[u]^{\Phi} &
	H^k_c\lrpar{W\rest{\pa X}} \ar[r]^-\delta \ar[u]^{\Phi_{\bar L}} &
	H^{k+1}_c\lrpar{W, W\rest{\pa X}} \ar[r] \ar[u]^{\wt \Phi} &
	H^{k+1}_c\lrpar{W} \ar[u]^{\Phi} \to \ldots}
\end{equation*}
where the connecting map $J$ is induced by
\begin{equation*}
	\xymatrix @R=1pt
	{\CI(X;\Lambda^k) \oplus \CI(\bbS W;\Lambda ^{k-1}) \ar[r] &
	\CI(\bar L;\Lambda^k) \oplus \CI(\bbS U;\Lambda ^{k-2}) \\
	(a, \alpha) \ar@{|->}[r] &
	( \pi^*_{\bar L} i^*_\pa a,\nu^L_*i^*_\pa \alpha ) }
\end{equation*}
$\Phi$ and $\Phi_{\bar L}$ are the isomorphisms defined in \eqref{MapPhi}
and \eqref{MapPhiL} respectively, and $\wt \Phi$ is the restriction of
$\Phi$ to the subcomplex of forms that vanish at $W\rest{\pa X}.$ 

The left and right squares above clearly commute, so we only 
check the commutativity of the middle square. Let $u$ be a $k$-form on
$W\rest{\pa X}$ with compact support, the map $\delta$ is induced by taking
any extension of $u$ into $W,$ say $e\lrpar u,$ and taking its exterior
derivative. It is convenient to find $\gamma\lrpar u$ such that
$\lrpar{\Phi e\lrpar u, \gamma\lrpar u, 0}$ is in $\cZ_*.$ To this end
choose a trivialization of $\bar L,$ denote by $t$ the fibre variable
along $\bar L,$ and note that
\begin{equation*}
\gamma(u)(t) = (-1)^{k-1} \int_0^t i^*_{\bar L} u
\end{equation*}
is as required and satisfies 
$d\gamma(u)=i^*_{\bar L}u-\pi^*_{\bar L}i^*_0 u$ (since $u$ is closed), 
and $\nu^{\bar L}_*\gamma(u) =0$ (since $\df i_{\pa_t} \gamma(u) =0).$

Thus 
\begin{equation*}
\begin{split}
	\wt\Phi\lrpar{\delta\lrpar u} 
	&= \lrpar{\Phi d e\lrpar u, 0, 0} 
	= \lrpar{ D_{\bbS W}\Phi e\lrpar u, 0, 0} \\
	&= D\lrpar{\Phi e\lrpar u, \gamma\lrpar u, 0} +
	\lrpar{ 0, 0, i^*_{\bar L} u, (-1)^{k-1}\nu^{\bar L}_* R^W_* u} \\
	&= D\lrpar{\Phi e\lrpar u, \gamma\lrpar u, 0} + \lrpar{ 0, 0, \Phi_{\bar L} u},
\end{split}
\end{equation*}
which shows that the induced maps in cohomology commute. It then follows
from the Five Lemma that the map $\wt \Phi$ is an isomorphism, i.e.,
$\cH^k\lrpar{\cZ_*;D_Z} \cong H^k_c\lrpar{W, W\rest{\pa X}}.$

To see that the complex $\cC_k\lrpar{\cT_*;D_T}$ represents $H^k_c\lrpar{W,
  W\rest{\pa X}}$ consider first the complex
\begin{gather*}
\CI\lrpar{X;\Lambda^*} \oplus 
\lrpar{ \CI\lrpar{\bbS W;\Lambda^*} \oplus \CI\lrpar{\pa X;\Lambda^*} }
\oplus \CI\lrpar{\bbS W\rest{\pa X};\Lambda^*} \\
\text{ with differential}
\begin{pmatrix} d & & & \\
-\pi^* & -d & & \\
i^* & & -d & \\
& i^* & \pi^* &d \end{pmatrix} 
\end{gather*}
which in view of \S\ref{rel-mwb} and \S\ref{SphereBundle} represents
$H^k_c\lrpar{W, W\rest{\pa X}}$ and then note that the complex
$\cC_k\lrpar{\cT_*;D_T}$ is obtained from this complex by applying the
push-forward along $\bar L$ which was shown in Lemma~\ref{rccbif.18} to be
an isomorphism.
\end{proof}

\section{Pseudodifferential operators and relative K-theory} \label{sec:Kthy}

In the standard case of Atiyah and Singer, the index of a vertical family
of Fredholm operators, $A$, acting on a superbundle $\bbE=E^+\oplus E^-$ on
the fibers of a fibration of closed manifolds $M \xrightarrow\phi B$ is
naturally thought of as an element of the topological K-theory group of
$B,$ e.g., 
\begin{equation*}
	\lrspar{\ker A} - \lrspar{\coker A} \in K(B)
\end{equation*}
when $\ker A$ (and hence $\coker A$) is a bundle over $B.$ On the other
hand $A$ itself, via its symbol
\begin{equation*}
	\sigma \lrpar A \in \CI\lrpar{\bbS^*M/B;\hom \bbE}
\end{equation*}
and the clutching construction, defines an element of $\Kc\lrpar{T^*M/B}$
and the index factors through this map 
\begin{equation*}
\xymatrix{
\Psi^*\lrpar{M/B;\bbE} \ar[rr]^{\Ind} \ar[rd]_{[a]}
&
&
K\lrpar B
\\
&
\Kc\lrpar{T^*M/B} \ar@{-->}[ru]_{\Ind_{\text{a}}}
&}.
\end{equation*}

One way to see this factorization is to start with a family $A$ as
above, say made up of pseudodifferentical operators of order zero, 
and eliminate properties that the index does not see.  That is, let
$\cK\lrpar{M/B}$ be the set of equivalence classes of vertical
pseudodifferential operators of order zero acting on superbundles $\bbE$ over $M$ where
two operators are considered equivalent if they can be connected by a
finite sequence of relations:
\begin{itemize}
\item [i)] $A\in \Psi^0\lrpar{M/B;\bbE} \sim B\in\Psi^0\lrpar{M/B;\bbF}$ if
there is a (graded) bundle isomorphism $\Phi:\bbE\to\bbF$ over $M$ such
that $B = \Phi^{-1} A \Phi,$
\item [ii)] $A\in \Psi^0\lrpar{M/B;\bbE} \sim \wt
  A\in\Psi^0\lrpar{M/B;\bbE}$ if $A$ and $\wt A$ are homotopic within
 elliptic operators, 
\item [iii)] $A\in \Psi^0\lrpar{M/B;\bbE} \sim
A\oplus\Id\in\Psi^0\lrpar{M/B;\bbE\oplus\bbC^{n|n}}$ where $\bbC^{n|n}$
is the trivial superbundle whose $\bbZ/2$ grading components are both
$\bbC^n.$ 
\end{itemize}
The resulting equivalence classes form a group, $\cK\lrpar{M/B}$ which can
be thought of as `smooth K-theory' and in this case is well-known to
coincide with the topological K-theory group $\Kc\lrpar{T^*M/B}.$ Indeed, 
the equivalence class of an operator only depends on its principal symbol 
\begin{equation*}
A \in \Psi^0\lrpar{M/B;\bbE} \implies
\sigma\lrpar A \in \CI\lrpar{S^*M/B;\hom \bbE}
\end{equation*}
in terms of which ($i$)-($iii$) give a standard `relative' definition of
$\Kc\lrpar{T^*M/B}.$  

Alternately, we can think of pseudodifferential operators of order $0$ 
as bounded operators acting on $L^2\lrpar{M/B}$ (defined,
e.g., using a Riemannian metric). These form a $*$-algebra and the closure
is a $C^*$-algebra, $\df A$, containing the compact operators, $\df K$. The
Fredholm operators are the invertibles in $\df A/\df K$, and the smooth
K-theory group described above is closely related to the odd $C^*$ K-theory
group of this quotient, $K_{C^*}^1\lrpar{\df A/\df K}.$ Indeed, the
principal symbol extends to a continuous map on $\df A,$
\begin{equation*}
  A\in \df A \implies \sigma \lrpar A \in \cC^0\lrpar{S^*M/B;\bbC},
\end{equation*}
which descends to the quotient $\df A/\df K$ (i.e., vanishes on $\df K$)
and allows us to identify the odd K-theory group with the stable homotopy
classes of invertible maps $S^*M/B \to \bbC.$ 

The most obvious difference between the smooth K-theory group and the $C^*$
K-theory group -- that the former is built up from smooth functions while
the latter from continuous functions -- disappears in the quotient. A more
significant difference comes from the way bundle coefficients are handled.
Stabilization allows us to replace the elliptic elements in
$\Psi^0\lrpar{M/B}$ with $\displaystyle
\lim_{\to}\mathrm{GL}_N\lrpar{\Psi^0\lrpar{M/B}},$ the direct limit of the
groups of invertible square matrices of arbitrary size and entries in
$\Psi^0\lrpar{M/B}.$ Note that an operator $A \in \Psi^0\lrpar{M/B;\bbE}$
acting on a superbundle $\bbE$ defines an element of the stabilization of
$\Psi^0\lrpar{M/B}$ by choosing a vector bundle $F$ such that $\bbC^j\cong
\bbE \oplus F$ (for some $j$) since then $\wt A=A \oplus \Id_F \in
\mathrm{GL}_j\lrpar{\Psi^0\lrpar{M/B}};$ a different choice of $F$ defines
the same element in $K^1_{C^*}\lrpar{M/B}.$ However if $\Phi$ is as in
($i$) above, it is possible that $\Phi^{-1} A \Phi$ and $A$ will define
distinct elements of $K^1_{C^*}\lrpar{M/B},$ and so the difference between
the smooth K-theory groups and the $C^*$ K-theory groups is essentially
that in the former we quotient out by ($i$) above. This is further pursued
in \cite[\S 2.3]{Albin-Melrose1}.

In the present paper we allow the fibers of the fibration $M\xrightarrow\phi B$ to have
boundary and consider three different calculi of
operators on a manifold with boundary, namely the scattering calculus, the zero
calculus, and the transmission calculus. In each case we work with the
smooth K-theory group as in the previous paragraph (denoted by
$\cK_{sc}\lrpar{M/B}$, $\cK_0\lrpar{M/B}$, and $\cK_{\tm}\lrpar{M/B}$
respectively) and, for the purposes of index theory, these groups contain all
of the relevant information. These groups have been shown to be isomorphic
to the topological K-theory group, $\Kc\lrpar{T^*M^\circ/B}$, in
\cite{RedBook} \cite{Melrose-Rochon} (scattering calculus),
\cite{Albin-Melrose1} ($0$-calculus), and \cite{BoutetdeMonvel1}
(transmission calculus). We briefly review what this entails.

\subsubsection*{Scattering calculus} A scattering operator $A \in
\Psi_{sc}^0\lrpar{M/B;\bbE}$ is determined up to a compact operator by its
image under two homomorphisms: the principal symbol $\sigma\lrpar
A\in\CI(\bbS^*M/B;\pi^*\hom \bbE)$ which is an homomorphism between the
lifts of $E^+$ and $E^-$ to $\bbS^*M/B$, and the boundary symbol
$b\in\CI(\comp{T^*M/B}_{\pa M};\pi^*\hom(\bbE))$ which is an homomorphism
between the lifts of $E^+$ and $E^-$ to the radial compactification of $T^*M/B$ over the
boundary. These symbols are equal on the common boundary of $\bbS^*M/B$
and $\comp{T^*M/B}_{\pa M},$ and so can be thought of as jointly
representing a section of $\hom(\bbE)$ lifted to the whole boundary of the
compact manifold with corners $\comp{T^*M/B}.$ An
operator is Fredholm on $L^2$ precisely when both of these symbols are
invertible, we call such an operator `fully elliptic'.

A fully elliptic operator $A$ can be deformed by homotopy within such
operators operators until $b,$ and $a$ in a neighborhood of the boundary,
are equal to a fixed bundle isomorphism. This isomorphism can then be used
to change the bundles so that $b$ is the identity and $a$ is the identity
near the boundary. This leads to an (arbitrary) invertible map into
$\hom(\bbE)$ that is the identity in an neighborhood of the boundary and
this is precisely the information that defines a relative K-theory class,
hence
\begin{equation}
\cK_{sc}\lrpar{M/B} = \Kc\lrpar{T^*M^\circ/B}.
\label{KThySc1}\end{equation}

The fully elliptic scattering operators of order zero with interior
symbol equal to the identity can be reduced by homotopy to perturbations of
the identity by scattering operators of order $-\infty.$ The smooth
K-theory of these perturbations of the identity is denoted
$\cK_{sc,-\infty}\lrpar{\pa M/B}$ and is readily seen to be equal to the
topological K-theory of the boundary,
\begin{equation}
\cK_{-\infty,sc}\lrpar{M/B} = \Kc\lrpar{T^*\pa M/B}.
\label{KThySc2}\end{equation}

\subsubsection*{Zero calculus} The analogues of \eqref{KThySc1} and
\eqref{KThySc2} also hold for the zero calculus, as shown in
\cite{Albin-Melrose1}. However where the scattering calculus is
`asymptotically commutative' as evinced in the boundary symbol
$b\in\CI(\comp{T^*M/B}_{\pa M};\pi^*\hom(\bbE))$, the zero calculus is
asymptotically commutative only in the directions tangent to the boundary
and is non-commutative in the direction normal to the boundary. Thus,
instead of a boundary symbol, the boundary behavior of a zero operator is
captured by a family of operators on a one-dimensional space, $\cI$, essentially
the compactified normal bundle to the boundary. This family, the reduced normal
operator
\begin{equation*}
	\cN\lrpar A \in \CI\lrpar{S^*\pa M/B;\Psi^0_{b,c}\lrpar{\cI;\bbE}},
\end{equation*}
takes values in the $b,c$ calculus ($b$ at one end of the interval, $c$ at
the other), and together with the interior symbol, determines the smooth
K-theory class of an operator $A\in\Psi^0_0\lrpar{M/B;\bbE}$.  A
description of the reduced normal operator is included in Appendix
\ref{NormalOp}.  As an element of the $b,c$ calculus, $\cN\lrpar A$ has
three model operators: its principal symbol, an indicial family at the
$b$-end, and an indicial family at the cusp end.  One can think of the
smooth K-theory as equivalence classes of zero pseudodifferential operators
or alternately as equivalence classes of invertible pairs
\begin{equation}
\begin{gathered} \lrpar{\sigma, \cN} \in
\CI\lrpar{S^*M/B;\hom\bbE} \oplus \CI\lrpar{S^*\pa M/B;
\Psi^0_{b,c}\lrpar{\cI;\bbE}}
\\
\Mst I_b\lrpar{\cN\lrpar{y,\eta}} =I_b\lrpar{\cN\lrpar{y,\eta'}},
\\
I_c\lrpar{\cN\lrpar{y,\eta}}\lrpar{\xi} =
    \sigma\lrpar{0,y,\frac{\xi}{\ang{\xi}},\frac{\eta}{\ang{\xi}}},
\\
\sigma\lrpar{\cN\lrpar{y,\eta}}\lrpar{\omega} = \sigma\lrpar{0,y,\omega,0}.
\end{gathered}
\label{ZeroPairs}\end{equation}

The isomorphism 
\begin{equation*}
	\cK_{-\infty,0}\lrpar{M/B} \cong \Kc\lrpar{T^*\pa M/B},
\end{equation*}
between the group of stable equivalence classes of invertible reduced
normal families of perturbations of the identity and a standard
presentation of $\Kc(T^*\pa M/B)$ comes from the contractibility of the
underlying semigroup of invertible operators on the interval
\cite{Albin-Melrose1}. Namely, this allows the reduced normal family to be
connected (after stabilization) to the identity through a curve of maps
$A(t)$ from $S^*\pa M/B$ into the invertible b-operators of the form
$\Id+A,$ $A$ of order $-\infty$ and non-trivial only at the one end of the
interval. The $b$-indicial family of $A(t)$ determines an invertible map
from $\comp{T^*\pa M/B}$ into $\hom\bbE$ equal to the identity at infinity
and hence an element of $\Kc\lrpar{T^*\pa M/B}$.

The contractibility of the group of invertible group of smooth
perturbations of the identity within the cusp calculus
\cite{Melrose-Rochon0} is used to show the isomorphism
\begin{equation*}
	\cK_0\lrpar{M/B} \cong \Kc\lrpar{T^*M^\circ/B}
\end{equation*}
between the group of stable equivalence classes of invertible pairs
\eqref{ZeroPairs} and the standard representation of
$\Kc\lrpar{T^*M^\circ/B}$.  Indeed, one can identify $E^+$ and $E^-$ near
the boundary, and, after a homotopy and a smooth perturbation, quantize by
a zero operator whose {\em full} $b$-indicial family is the identity, and then
use the contractibility to take the reduced normal operator to the
identity. The principal symbol of the resulting operator is equal to the
identity near the boundary and classically defines an element of
$\Kc\lrpar{T^*M^\circ/B}$.

\subsubsection*{Transmission calculus}

A smooth K-theory class in the Boutet de Monvel calculus is an equivalence
class of operators of the form
\begin{equation}
	\cA=\begin{pmatrix}
	    	\gamma^+A+B &  & K \\ \\ T & & Q 
	    \end{pmatrix}
	:\begin{matrix}
	 	\CI\lrpar{X;E^+} & & \CI\lrpar{X;E^-} \\ \oplus & \to & \oplus \\ 
		\CI\lrpar{\pa X;F^+} & &\CI\lrpar{\pa X;F^-} 
	 \end{matrix}
\label{BdMmatrix}\end{equation}
acting on the superbundles $\bbE$ over $X$ and $\bbF$ over $\pa X$.  The
principal symbol of $A$ is required to satisfy the {\em transmission
  condition} at the boundary, and in particular this forces the order of
$A$ to be an integer. We can assume without loss of generality that the
order of $A$ and its `type' are both zero (see \cite{BoutetdeMonvel1},
\cite{Grubb}).

The boundary behavior of $A$ is modeled by a family of Wiener-Hopf
operators parametrized by the cosphere bundle over the boundary. For $A$ as
above we denote its boundary symbol by
\begin{equation}
	N\lrpar A\lrpar{y,\eta}=\begin{pmatrix}
	    	h^+p+b &  & k \\ \\ t & & q 
	    \end{pmatrix}
	:\begin{matrix}
	 	\CI\lrpar{X;H^+\otimes E^+_y} & & \CI\lrpar{X; H^+\otimes E^-_y} 
			\\ \oplus & \to & \oplus \\ 
		\CI\lrpar{\pa X;F^+_y} & &\CI\lrpar{\pa X;F^-_y} 
	 \end{matrix}
\label{BdyBoutet}\end{equation}
where $h^+$ is the projection from the space of functions
$\CI\lrpar{\bbR;\bbC}$ that have a regular pole at infinity to the subspace
$H^+$ of those that vanish at infinity and can be continued analytically to
the lower half-plane.

Boutet de Monvel showed that any operator $A$ whose principal symbol and
boundary symbol are both invertible is homotopic through such operators to
one of the form
\begin{equation*}
	\begin{pmatrix}
	    	\gamma^+\wt A &  0 \\ 0 & \wt Q 
	    \end{pmatrix}
\end{equation*}
where $\wt A$ is equal to the identity near the boundary. Thus
$\cK_{\tm}\lrpar{M/B}$ clearly surjects onto $\Kc\lrpar{T^*M^\circ/B}$ which
is all we will need. However we point out that, using the description of the $C^*$-algebra
K-theory from, e.g., \cite{MNS}, \cite{MSS}, and the comparison with smooth
K-theory in \cite[\S 2.3]{Albin-Melrose1}, it is possible to show that
$\cK_{\tm}\lrpar{M/B}$ is actually equal to $\Kc\lrpar{T^*M^\circ/B}$.

\section{Chern character and the families index theorem}\label{sec:Chern}

The Chern character is a homomorphism from $\Ch:K(X)\longrightarrow
H^{\even}(X)$ which gives an isomorphism after tensoring with $\bbQ.$ Chern-Weil
theory gives a direct representation of the Chern character in deRham
cohomology for a superbundle $\bbE.$ If $\nabla$ is a graded connection on
$\bbE,$ i.e.\ a pair of connections $\nabla^\pm$ on $E^\pm$ with curvatures
$(\nabla^\pm)^2=-2\pi i \omega_\pm$ then
\begin{equation}
\Ch\lrpar\bbE = e^{\omega_+} - e^{\omega_-}. 
\label{rccbif.121}\end{equation}
Different choices of connection are homotopic and give cohomologous closed forms.

In the case of a compact manifold with boundary $X,$ the K-theory with
compact support in the interior, denoted here $\Kc(X\setminus\pa X)$, is
represented by superbundles $\bbE$ where $E^\pm=\bbC^N$ near the
boundary. Then the same formula, \eqref{rccbif.121}, gives the relative
Chern character $\Ch:\Kc(X\setminus\pa X)\longrightarrow
\Hc^{\even}(X\setminus\pa X)$ provided the connections are chosen to reduce to
the trivial connection, $d,$ near the boundary.

There are natural isomorphism $\Kc(X\setminus\pa
X)\longrightarrow\operatorname{K}(X,\pa X)$ and $\Hc^{\even}(X\setminus\pa
X)\longrightarrow H^{\even}(X,\pa X)$ with the corresponding relative
objects. Chains for $\Kc(X\setminus\pa X)$ are given by pairs $(\bbE,a)$ of
a superbundle over $X$ and an isomorphism $a:E^+\longrightarrow E^-$ over
$\pa X.$ In \cite{MR1401125} Fedosov gives an explicit formula for the
Chern character, in cohomology with compact supports of the cotangent
bundle, of the symbol of an elliptic operator acting between vector
bundles. This can be modified to give the relative Chern character in this
setting with values in the chain space discussed in Lemma~\ref{rccbif.12}
\begin{equation}
\begin{gathered}
\Ch(\bbE,a)=(\Ch(\bbE),\tCh(a))
\\
\tCh(a)=-\frac1{2\pi i}
	\int_0^1 \tr\left(a^{-1}(\nabla a)e^{w(t)}\right)dt,
\\
\Mwhere
w(t)=(1-t)\omega _++ta^{-1}\omega _-a + \frac 1{2\pi i} t(1-t)(a^{-1}\nabla a)^2.
\end{gathered}
\label{rccbif.22A}\end{equation}
Here the boundary term, $\tCh(a),$ is a `regularized' (or improper) form of
the odd Chern character.

In the context of the index formula this also corresponds rather naturally
to the relative cohomology as discussed above. Essentially by
reinterpretation we find 

\begin{proposition}[Fedosov \cite{MR1401125}]\label{Z-i.7} If
$\pi:U\longrightarrow X$ is a real vector bundle, $\bbE\longrightarrow X$
is a superbundle and $a\in\CI(\bbS U ;\pi^*(\hom(\bbE))$ is elliptic (\ie
invertible) then for any graded connection on $\bbE,$ with curvatures
$\Omega_\pm=-2\pi i \omega _\pm$ and induced connection $\nabla$ on
$\hom(\bbE),$ the class
\begin{equation}
\Ch(U,\bbE,a)=(\Ch(\bbE),\tCh(a))\in
	\CI(X;\Lambda^{\even})\oplus\CI(\bbS U;\Lambda^{\odd}),
\label{Z-i.8}\end{equation}
given by the formul\ae \eqref{rccbif.121} and \eqref{rccbif.22A},
represents the relative Chern character  
\begin{equation}
	\Kc^0(U)\longrightarrow \cH^{\odd}(\bbS U,\pi)\simeq \Hc^{\even}(U).
\label{Z-i.89}\end{equation}
\end{proposition}

\begin{proof} 
That the pair $(\Ch(\bbE), \tCh(a))$ is $D$-closed and gives a
well-defined class in the cohomology theory follows in essence as in
standard Chern-Weil theory. We include such an argument for completeness
and for subsequent generalization.

Set $\theta=a^{-1}\nabla a$ and note that the connection $\wt\nabla =
\nabla + \lrspar{t\theta,\cdot}$ has curvature
\begin{equation*}
	\Omega\lrpar t
	= -2\pi i\omega_+ + \lrpar{ t\nabla\theta +t^2\theta^2}
	= -2\pi i \omega(t)
\end{equation*}
It follows that $\wt\nabla e^{ \omega(t) } =0$ and hence
\begin{equation*}
\nabla e^{\omega(t)} = \wt\nabla e^{\omega(t)} - t\lrspar{\theta, e^{\omega(t)}}
	=t\lrspar{e^{\omega(t)},\theta}.
\end{equation*}
Thus
\begin{equation}\begin{split}
d\tCh(a)
&=-\frac1{2\pi i}\int_0^1 \tr\nabla\lrpar{\theta e^{\omega(t)}}\;dt\\
&=-\frac1{2\pi i}\int_0^1\tr\lrpar{ \lrpar{\nabla\theta} e^{\omega(t)} -
	\theta\nabla e^{\omega(t)}}\;dt\\
&=-\frac1{2\pi i}\int_0^1\tr\lrpar{ \lrpar{\nabla\theta} e^{\omega(t)} -
	t\theta\lrspar{e^{\omega(t)},\theta}}\;dt\\
&=-\frac1{2\pi i}\int_0^1\tr\lrpar{ \lrpar{\nabla\theta+2t\theta^2} e^{\omega(t)} 
		- t\lrspar{\theta e^{\omega(t)},\theta}}\;dt
\label{Z-i.58}\end{split}\end{equation}
and since the trace vanishes on graded commutators,
\begin{equation}
d\tCh(a)=\int_0^1 \tr\left(\frac{d \omega(t)}{dt}e^{\omega(t)}\right)dt
=\pi^*\tr(e^{\omega_-})-\pi^*\tr(e^{\omega_+})
=-\pi^*\Ch(\bbE).
\label{Z-i.57}\end{equation}

As is well-known (and explained in \S\ref{sec:HtpyInv}), this same formula
shows independence of the connection and homotopy invariance.  That this
class actually represents the (appropriately normalized) Chern character
follows from Fedosov's derivation in \cite{MR1401125}.  \end{proof}

For an elliptic family of pseudodifferential operators
$A\in\Psi^0(M/B;\bbE)$ where $\psi:M\longrightarrow B$ is a fibration with
typical fibre $Z$ and $\bbE$ is a superbundle over $M,$ the symbol $\sigma
(A)\in\CI(S^*(M/B);\pi^*\hom(\bbE))$ is invertible (by assumption) and the
discussion above applies with $U=T^*(M/B),$ the fibre cotangent bundle. The
index formula of Atiyah and Singer is then given as a composite 
\begin{equation}
\xymatrix{ 
	\Kc(T^*(M/B)) \ar@{-->}[r]^{\ind} \ar[d]^{\Ch} & H^{\even}\lrpar B\\
	\cH^{\odd}(S^*(M/B),\pi) \ar[r]^{\wedge\Td(Z)}&
	\cH^{\odd}(S^*(M/B),\pi)\ar[u]^{\int_{S^*Z}}}
\label{Z-i.26}\end{equation}
So, for such a family of operators,
\begin{equation}
\Ch(\ind(A))=\int_{S^*Z}\Ch(T^*(M/B),\bbE,\sigma(A))\wedge\Td(Z)
\label{Z-i.13}\end{equation}
where $\Td(Z)$ is the Todd class of $Z.$ That this is well-defined follows
from \eqref{FibPush}.

\subsection{Scattering families index theorem}

Consider next a fibration of manifolds with boundary $M \xrightarrow{\psi}
B.$ As described in \S\ref{sec:Kthy} the compactly supported K-theory of
$W=T^*M^\circ/B$ can be represented by scattering operators. The K-theory
class of a scattering operator is determined by its two symbol maps, its
principal symbol and its boundary symbol. We now explain how to represent
the Chern character of the corresponding K-theory class in terms of this data.

In fact given any manifold with boundary $M$ and a bundle $W \to M$ any
class in the compactly supported K-theory of $W\rest{M\setminus \pa M}$ can
be represented by a superbundle $\bbE \to M$ and two invertible maps
\begin{equation*}
	a \in \CI(\bbS^*W;\pi^*\hom \bbE),\
	b \in \CI(\comp{W}_{\pa M};\pi^*\hom \bbE).
\end{equation*}
We recall, from \S\ref{rel-scat}, that one may compute the cohomology
$H^*_c(W\rest{M\setminus \pa M})$ via the complex
\begin{gather*}
\CI(M;\Lambda^k) \oplus
\left\{(\alpha,\beta);\alpha\in\CI(\bbS W;\Lambda ^{k-1}),\ 
\beta\in\CI(\comp{W}_{\pa M};\Lambda ^{k-1})\Mand
\iota _{\pa}^*\alpha=\iota _{\pa}^*\beta\right\},
\\
D = \begin{pmatrix}
	d & 0 \\ \phi & -d  
\end{pmatrix},
\Mwhere
\phi = \begin{pmatrix} -\pi^*
\\
-\pi^*i^* \end{pmatrix}.
\end{gather*}
Then, choosing a graded connection on $\bbE$, the forms, again using
\eqref{rccbif.121} and \eqref{rccbif.22A},
\begin{multline}
\Ch(W,\bbE,a,b)=\lrpar{ \Ch\lrpar \bbE,\tCh(a),\tCh(b) } \\
\in \CI\lrpar{X;\Lambda^\even}
\oplus \CI\lrpar{\bbS W;\Lambda^\odd}
\oplus \CI\lrpar{\comp W_{\pa M}; \Lambda^\odd},
\label{Z-i.52}\end{multline}
represent the Chern character of the K-theory class associated to
$\lrpar{W, \bbE, a, b}.$

\begin{proposition}\label{ScatChern} 
For $W$, $\bbE$, $a$, and $b$ as
above, $\Ch(W,\bbE,a,b)$ is $D$-closed and its relative cohomology class
coincides with the Chern character of the K-theory class defined by
$(W,\bbE,a,b).$ 
\end{proposition}

\begin{proof}
From the definition of the differential in \S\ref{rel-scat}, $D$-closed means that
\begin{gather*}
	d_X\Ch\lrpar \bbE =0,\
	d_{\bbS W}\tCh\lrpar a = -\pi^*\Ch\lrpar \bbE\Mand
	d_{\bar W_{\pa X}}\tCh\lrpar b = -\pi^*i^*\Ch\lrpar \bbE.
\end{gather*}
Thus that the putative Chern character is $D$-closed and homotopy invariant
follow just as in Lemma \ref{Z-i.7}. It is also invariant under changes of
$\bbE$ by stabilization and bundle isomorphism since this is true of the
forms $\lrpar{ \Ch\lrpar \bbE,\tCh(a),\tCh(b) }$ themselves, and hence it only
depends on the K-theory class.

As explained in \cite{RedBook} and reviewed in \S\ref{sec:Kthy}, there is a
representative of the K-theory class with $b=\Id$ and $a=\Id$ near the
boundary, and, since in this case \eqref{Z-i.52} coincides with Fedosov's
formula \eqref{Z-i.8}, we conclude that \eqref{Z-i.52} is the usual Chern
character map.
\end{proof}

The index formula follows similarly. Given a vertical family of fully
elliptic scattering pseudodifferential operators acting on a superbundle
$\bbE \to M,$ it is possible to make a homotopy within fully elliptic
scaterring operators until the symbols are bundle isomorphisms at and near
the boundary. Thus a formula which is homotopy invariant and which
coincides with the usual Atiyah-Singer index formula when the operators are
trivial at the boundary must give the index.

\begin{proposition}\label{Z-i.67} The index in cohomology for a family of
fully elliptic scattering pseudodifferential operators on the fibres of a
fibration is given by the Atiyah-Singer formula essentially as in
\eqref{Z-i.26}, \eqref{Z-i.13}:
\begin{equation}\begin{split}
\ind(A)&=\int \Ch(\bbE,\sigma (A),\beta(A))\wedge \Td(Z)\\
	&=\int_{\scS^*Z}\tCh\lrpar{\sigma (A)}\wedge \Td(Z)
	+\int_{\comp{\scT^*_{\pa Z}Z}}\tCh\lrpar{\beta(A)}\wedge
\iota_{\pa}^*\Td(Z)
\label{Z-i.53}\end{split}\end{equation}
where $\Td(Z)\in\CI(M;\Lambda^{\even})$ is a deRham form representing the
Todd class of the fibres of $\psi:M\longrightarrow Y.$
\end{proposition}

\begin{proof} Homotopy invariance follows from Proposition \ref{ScatChern}
  and when the family of operators are trivial near the boundary this
  clearly coincides with the Atiyah-Singer families index formula.
\end{proof}

\subsection{Zero families index theorem}\label{sub:ZeroFIT}

The non-commutativity of the boundary symbol, in the normal direction,
makes the derivation of a formula more challenging, so we first consider
the simple case where only the boundary symbol appears.

\subsubsection*{Perturbations of the identity}

Consider a fibration as in \eqref{rccbif.3} where the typical fiber, $X,$
is a manifold with boundary, denoted $Y.$ The restriction of the fibrewise
cotangent bundle to the boundary $T^*_{\pa M}M/B$ has a trivial line
sub-bundle, the conormal bundle, with the quotient being $T^*\pa M/B.$ Thus
as explained in \S\ref{SphereBundle} the compactly supported cohomology of
$U=T^*\pa M/B$ can be realized as the cohomology $\cH^k\lrpar{\bbS U,\pi}$
of the complex
\begin{equation}
\CI(\pa M;\Lambda ^*)\oplus\CI(\bbS U;\Lambda ^{*-1}),\
D=\begin{pmatrix}d&0\\ -\pi^*&-d \end{pmatrix}.
\end{equation}

On the other hand, as explained in \S\ref{sec:Kthy}, the zero
pseudodifferential operators on $X$ can be used to realize the K-theory of 
$Y.$ Indeed
\begin{equation}
\Kc\lrpar{T^*\pa M/B} = \cK_{0,-\infty}\lrpar{M/B},
\label{rccbif.22}\end{equation}
and, since the class of an operator $\Id + A$ (acting as an odd operator on
sections of the superbundle $\bbE$) in the quotient on the right is
determined by its reduced normal operator, any class in
$\Kc\lrpar{T^*\pa M/B}$ can be represented by a map $N \in \CI\lrpar{\bbS
  U,\BPs{-\infty}{[0,1];\bbE} }.$ 

A natural candidate for the Chern character, in view of the previous
sections, would be to choose a graded connection on $\bbE$ and then
consider $\wt \Ch (N).$ However, elements of $\BPs{-\infty}{[0,1];\bbE}$
are generally not trace-class and this corresponding expression is not well
defined unless the trace is renormalized. In this setting, the `$b$-trace'
\begin{equation*}
	\bTr: \BPs{-\infty}{[0,1];\bbE} \longrightarrow \bbC
\end{equation*}
is an extension of the trace that however is not itself a trace. Indeed, instead of vanishing on commutators it satisfies a `trace-defect formula', namely
\begin{equation}
  \bTr\lrpar{\lrspar{A,B}} = \frac1{2\pi i}\int_\bbR \Tr\lrpar{\frac{\pa
      a}{\pa \xi} b} \;d\xi,
\label{BTraceDefect}\end{equation}
where $a$ and $b$ are the indicial operators of $A$ and $B$ respectively (see \cite[Chapter 4]{APSBook}).
Renormalized traces are briefly discussed in Appendix \ref{sec:ResTrace} following \cite{Melrose-Nistor0}.

Given a graded connection $\nabla$ on $\bbE$ and a choice of boundary defining
function $x$ (on which the definition of the $b$-trace depends) consider
the algebra of matrices with entries which are $b$-pseudodifferential
operators on an interval. The $b$-trace will be used to define odd `eta
forms' on the semigroup of $L^2$-invertible order $-\infty$ perturbations
of the identity. These are regularized versions of the odd Chern character
given by \eqref{rccbif.22A}. Taking into account the fact that $\bTr$ is
not a trace set
\begin{multline}
	\eta^\odd_{\bo}\lrpar N = -\frac1{2\pi i} \int_0^1 
	\lrpar{1-t}\bTr\lrpar{ N^{-1}\lrpar{\nabla N}e^{w_N(t)}} 
		+ t\bTr\lrpar{\lrpar{\nabla N}e^{w_N(t)}N^{-1}} \;dt\\
	\Mwhere 
	w_N(t)=(1-t)\omega _+ + tN^{-1}\omega _-N 
		+ \frac1{2\pi i}t(1-t)(N^{-1}\nabla N)^2.
\label{RenEta}\end{multline}

The {\em even} Chern character of the indicial family of $N$ is also needed
here, since these are in essence loops into smoothing operators. The even
Chern character therefore arises from $\tCh$ by transgression
\begin{equation}\label{EvenChern}
\Ch^\even\lrpar a
= \int_\bbR\df i_{\pa_\xi} \tCh_{Y\times\bbR}\lrpar a \; d\xi.
\end{equation}

\begin{lemma}\label{lem:dEta}
If $N$ is a family of a Fredholm zero operators of the form $\Id+A$ where 
$A \in \Psi^{-\infty}_0\lrpar{M/B;\bbE}$ and $a=I_b\lrpar N$ then
\begin{equation}\label{dEta}
d\eta^\odd_{\bo}\lrpar N
	= -\pi^*\Ch\lrpar \bbE  + \nu^{\bar L}_* \tCh\lrpar{I_b(N)}
	= -\pi^*\Ch\lrpar \bbE  + \Ch^{\even}\lrpar a 
\end{equation}
where $\Ch\lrpar \bbE$ is given by \eqref{rccbif.22A}.
\end{lemma}

\begin{proof} Set
\begin{equation*}
	\theta = N^{-1}\nabla N,\ \theta_a = a^{-1}\nabla a\Mand
	\theta_\xi = a^{-1}\frac{\pa a}{\pa \xi},
\end{equation*}
and recall, as in Proposition~\ref{Z-i.7}, that
\begin{equation*}
\pa_te^{w\lrpar t} 
=-\frac1{2\pi i}\lrpar{\nabla\lrpar{\theta e^{w\lrpar t}} 
	+ t\theta e^{w\lrpar t}\theta + t\theta^2 e^{w\lrpar t} }.
\end{equation*}

Then 
\begin{multline*}
d\eta\lrpar N
= d\lrspar{
 \frac i{2\pi}\int_0^1 
\bTr\lrpar{ \lrpar{1-t} \theta e^{w_N(t)}
	+ t N\theta e^{w_N(t)}N^{-1}} \;dt }\\
=\frac i{2\pi}\int_0^1 
	\bTr\lrpar{\nabla\lrpar{ \lrpar{1-t} \theta e^{w_N(t)}
		+ t N\theta e^{w_N(t)}N^{-1}} }\;dt \\
	=\int_0^1 
\bTr\lrpar{ t N \pa_te^{w_N(t)} N^{-1} + (1-t) \pa_t e^{w_N(t)}} \;dt
\\
	+\frac i{2\pi}\int_0^1 
\bTr\lrpar{ t\lrpar{1-t}
\lrspar{N\lrspar{\theta e^{w_N(t)}, \theta}, N^{-1}} }\;dt
\\
=\int_0^1 \pa_t \bTr\lrpar{ t N e^{w_N(t)} N^{-1} + (1-t) e^{w_N(t)}} \;dt
\\
+ \int_0^1\bTr\lrpar{\lrspar{Ne^{w_N(t)},N^{-1}}} \;dt
+\frac i{2\pi}\int_0^1\bTr
\lrpar{ t\lrpar{1-t} \lrspar{N\lrspar{\theta e^{w_N(t)}, \theta}, N^{-1}} }\;dt.
\end{multline*}
The integral of the $t$-derivative reduces to
\begin{equation*}
	\bTr\lrpar{ N e^{w_N(1)} N^{-1}} - \bTr\lrpar{e^{w_N(0)}}
	= -\pi^* \Ch\lrpar \bbE
\end{equation*}
and the other terms can be evaluated using the trace-defect formula giving
\begin{equation*}
\int_0^1 \frac i{2\pi} \int_{\bbR} \tr\lrpar{ \theta_\xi e^{w_a(t)}} \;d\xi \;dt
+\frac i{2\pi}\int_0^1 \frac i{2\pi} t\lrpar{1-t} \int_{\bbR} 
\tr\lrpar{ -\theta_\xi \lrspar{\theta_a e^{w_a(t)}, \theta_a} } \;d\xi \;dt.
\end{equation*}
Comparing this with 
\begin{equation*}
\begin{split}
\df i_{\pa_\xi} \tCh&\lrpar a = 
\frac i{2\pi}\df i_{\pa_\xi} \int_0^1 \tr\lrpar{ \theta_a e^{w_a\lrpar t} } \;dt
\\
&=\frac i{2\pi}\int_0^1
\tr\lrpar{ \theta_\xi e^{w_a\lrpar t}- \frac i{2\pi}\theta_a e^{w_a\lrpar t}
\lrpar{t\pa_\xi\lrpar{ \theta_a} - t\nabla
\lrpar{\theta_\xi} -t^2\lrspar{\theta_a,\theta_\xi}}} \;dt
\\
&=\frac i{2\pi}\int_0^1
\tr\lrpar{ \theta_\xi e^{w_a\lrpar t}-\frac i{2\pi}\theta_a e^{w_a\lrpar t}
\lrpar{t \lrpar{1-t} \lrspar{\theta_a,\theta_\xi} }} \;dt
\end{split}\end{equation*}
yields \eqref{dEta}.
\end{proof}

A fundamental property of the reduced normal operator is that its
$b$-indicial symbol only depends on the base $\pa M/B$ and not on the
cotangent variables, so $\Ch^\even(I_b(N))$ can be regarded as a form on $\pa
M.$ Thus the result of the lemma shows that the forms
\begin{equation*}
\Ch\lrpar\bbE - \Ch^\even\lrpar a \in \CI\lrpar{\pa M;\Lambda^\even},\
\eta\lrpar N \in \CI\lrpar{\bbS U;\Lambda^\odd}
\end{equation*}
define a class in $\cH^\odd\lrpar{\bbS U, \pi}.$

\begin{proposition}\label{prop:OddZeroChern} The (odd) relative Chern
character for $\Kco(\ZT^*_{\pa M}M/B)$ of the (zero) cotangent bundle of a compact
manifold with boundary may be realized in terms of the reduced normal
operators by
\begin{equation}
\Ch^\odd(P)=\lrpar{\Ch\lrpar\bbE - \Ch^\even[I(\RN(P))],\eta_{\bo}( \RN(P) )}, 
\label{27.9.06.8}\end{equation}
as a class in $\cH^{\odd}(S^*\pa M/B,\pi)$ for $P \in \Id+\Psi_0^{-\infty}(M/B; \bbE)$ Fredholm on $L^2_0(M/B)$.
\end{proposition}

{\bf Remark.} The proof of this proposition is more complicated than that
of Proposition \ref{ScatChern} since the isomorphism \eqref{rccbif.22} does
not allow us to represent a K-theory class by an operator that is trivial
at infinity (see \S\ref{sec:Kthy}).

\begin{proof} Since $\Ch^\odd(P)$ defines a class in $\cH^{\odd}(S^*\pa M/B,\pi)$
it remains only to show that this is the relative Chern character. To this end,
we will show that the diagram
\begin{equation}\label{ChernDiag}
\xymatrix{\cK_{-\infty,0}^0\lrpar{M/B} \ar[r]_\cong \ar[d]^{\Ch} 
& \Kc^0\lrpar{T^*\pa M/B} \ar[d]^{\Ch}
\\
\cH^\odd\lrpar{S^*\pa M/B,\pi} \ar[r]_\cong & H_c^\even\lrpar{T^*\pa M/B} }
\end{equation}
commutes. It is convenient to represent the compactly supported $K$-theory of $T^*\pa M/B$ using a classifying group
\begin{equation*}
	\Kc^0\lrpar{T^*\pa M/B}
	= \lim_{\to} \lrspar{X; C^{\infty}\lrpar{\lrpar{\mathbb{S}^1, 1}; \lrpar{\mathrm{GL}\lrpar{N}, \Id}}} 
	= \lrspar{X; G^{-\infty} },
\end{equation*}
where $G^{-\infty}$  is the group of invertible `suspended' operators on a closed manifold that differ from the identity by a smoothing operator (see \cite{Melrose-Rochon0}).
 
The isomorphism at the top of \eqref{ChernDiag} between the group of stable
equivalence classes of invertible reduced normal families of perturbations
of the identity and a standard presentation of $\Kco(\ZT^*_{\pa M}M/B)$
comes from the contractibility of the underlying semigroup of invertible
operators on the interval. Namely, this allows the reduced normal family to
be connected (after stabilization) to the identity through a curve of maps
$A(t)$ from $S^*\pa M/B$ into the invertible b-operators of the form
$\Id+A,$ $A$ of order $-\infty$ and non-trivial only at the one end of the
interval.  In fact we can easily arrange for the family $A(t)$ to be
constant near the two endpoint of the parameter interval.

The indicial family of this curve initially only depends on the base variables $\pa M/B$ and
the indicial parameter, so can be interpreted as fixing a
homotopy class of smooth maps
\begin{equation*}
\alpha:\ZT^*\pa M/B\longrightarrow G^{-\infty}
\end{equation*}
and hence an element of $\Kc^0(\ZT^*\pa M/B).$
 
The cohomology class of the resulting Chern character defines, via the map
\eqref{MapPhi}, an element of $\cH^\odd\lrpar{S^*\pa M/B,\pi}$.  In this
case, this map consists of passing from $\Ch\lrpar\alpha$ back to
$\Ch\lrpar{A(t)}$, decomposing this as
\begin{equation*}
\Ch\lrpar{A(t)} = \Ch\lrpar{A_t} + \Ch'\lrpar{A_t} \wedge dt
\end{equation*}
with both $\Ch\lrpar{A_t}$ and $\Ch'\lrpar{A_t}$ independent of $dt$, and
then keeping
\begin{equation}
\lrspar{\lrpar{\Ch\lrpar{A_0},-\int_0^\infty \Ch'\lrpar{A_t}\;dt}}\in
  \cH^\odd\lrpar{S^*\pa X,\pi}.
\label{ChImage}\end{equation}
The normalization of the Chern character $\Ch^\even(a)$ is fixed by the
requirement that it \emph{be} the usual multiplicative map
\begin{equation*}
\Ch:K^0(\pa M/B)\longrightarrow H^{\even}(\pa M/B).
\end{equation*}
Hence the first term in \eqref{ChImage} coincides with
$\Ch^\even\lrpar a$ defined in \eqref{EvenChern}, while the second term
$\omega$ satisfies $d\omega = -\pi^*\Ch^\even\lrpar a$. It follows that
\eqref{ChImage} and \eqref{27.9.06.8} define the same class in
$\cH^\odd\lrpar{S^*\pa X,\pi}.$
\end{proof}

\subsubsection*{General Fredholm zero operators}

Again consider a fibration \eqref{rccbif.3}, where the fibers $X$ are
manifolds with boundary. The compactly supported $K$-theory of $T^*M/B$
relative to the boundary can be represented by stable homotopy classes of
Fredholm operators in the zero calculus. Thus each class is represented by
a superbundle $\bbE$ and a pair of maps, 
\begin{equation*}
\lrpar{\sigma, \cN} 
\in \CI\lrpar{S^*M/B;\hom\bbE} 
\oplus \CI\lrpar{S^*\pa M/B; \Psi^0_{b,c}\lrpar{\cI;\bbE}}
\end{equation*}
as in \eqref{ZeroPairs}. The vector bundle $W=T^*M/B$ has a trivial line
sub-bundle at the boundary $L$ (the normal bundle to the boundary) so
$H^*_c\lrpar{W^\circ}$ can be realized using forms on $X,$ $\bbS W,$ $\bar 
L,$ and $\bbS U$ (with $U=W\rest{\pa M}/L).$ This will allow us to write a
formula for the Chern character involving only $\bbE,$ $\sigma,$ $\cN$ and
a choice of graded connection on $\bbE.$ To simplify the discussion the map
$\sigma$ restricted to the inward pointing end of $L$ will be used to
identify $E^+$ and $E^-$ over the boundary. Then the connections can be
chosen to be compatible with this identification and the result is that we
can arrange
\begin{quote}
the Chern character of $\bbE$ restricted to the boundary is identically
  zero.
\end{quote}

The main novelty in the construction of the explicit Chern character for
general Fredholm families in the zero calculus involves the eta term coming
from the reduced normal operator. This is now a doubly-regularized form, in
the sense that the divergence at the boundary needs to be removed as before
but there is also divergence coming from the fact that these operators
are now not of trace class even locally in the interior of the interval.
Indeed, an operator in the $b,c$ calculus on the interval will be of trace
class if and only if it has order less than $-1$ and its kernel vanishes at
the boundary. This regularization is discussed in
Appendix~\ref{sec:ResTrace} (extending \cite{Melrose-Nistor0}) using
complex powers of an admissible operator $Q$ in the $b,c$ calculus and a
(total) boundary defining function $x.$

As discussed in Appendix~\ref{NormalOp}, the choice of a metric at the
boundary trivializes the interval bundle over the boundary on which the
reduced normal operator acts. In particular the reduced normal operator of
an element of order $0$ in the zero calculus then becomes a well-defined
smooth map $\cN:S^*\pa
M/B\longrightarrow\Psi_{\fbsc}^0\lrpar{\cI;\bbE\rest{\pa M}}.$ The total
symbol of $Q$ is a function on the cotangent bundle of the interval
$T^*\cI=\cI_r\times\bbR_\omega$ and it depends only on the cotangent
variable $\omega$, not on $r.$

Generalizing \eqref{RenEta}, for any element of the zero calculus with
invertible reduced normal family (on $L^2)$, set
\begin{multline}\label{EtaBC}
 \eta_{\fbsc}^\odd\lrpar\cN \\
 =\frac i{2\pi} \int_0^1 
 \lrpar{1-t}\RTr_{\fbsc}\lrpar{\cN^{-1}\lrpar{\nabla\cN}e^{\omega_{\cN}\lrpar t}}
 +t\RTr_{\fbsc}\lrpar{\lrpar{\nabla\cN}e^{\omega_{\cN}\lrpar t}\cN^{-1}} \;dt
\end{multline}
where the inverse takes values in the large calculus. 

The differential of \eqref{EtaBC} will involve the even Chern characters of the indicial families at the $b$ and $cusp$ ends. These are defined by 
\begin{equation*}
\Ch^\even\lrpar a
= \int_\bbR\df i_{\pa_\xi} \tCh_{Y\times\bbR}\lrpar a \; d\xi
\end{equation*}
with $a=I_b\lrpar\cN$ or $a=I_c\lrpar\cN.$ Initially one might expect an
integrability issue because both of these are homogeneous of degree
zero. However at least one of the factors in the integrand will involve
$\frac{\pa a}{\pa \xi},$ and since in either case $a$ has an expansion
\begin{equation*}
	a\sim a_0 + a_1\xi^{-1} + a_2\xi^{-2} + \ldots
\end{equation*}
as $\xi\to\infty,$ the integrand must vanish to second order at infinity
and hence is integrable.

\begin{lemma}\label{Z-i.81} 
For an element $A\in\ZPs0{X;\bbE}$ with reduced normal operator
$\cN:X\longrightarrow\Psi_{\fbsc}^0\lrpar{\cI,\bbE\rest{\pa X}}$ taking
values in the $L^2$-invertible operators,
\begin{equation}
\begin{gathered}
d\eta_{\fbsc}^\odd\lrpar\cN
=-\pi^*i^*_{\pa X}\Ch\lrpar\bbE +
\Ch^\even\lrpar{I_b\lrpar\cN}
-\Ch^\even\lrpar{I_c\lrpar\cN} \\
\Ch^\even\lrpar{I_b\lrpar\cN}
-\hat\nu^L_* i_{\pa X}^* \wt\Ch\lrpar{\sigma\lrpar A}
\end{gathered}
\label{Z-i.80}\end{equation}
\end{lemma}

\begin{proof} 
The computation in the proof of Lemma~\ref{lem:dEta} shows that
\begin{multline*}
d\eta_{\fbsc}^\odd\lrpar\cN
=-\pi^*\Ch\lrpar\bbE 
+ \int_0^1\RTr\lrpar{\lrspar{Ne^{w_N(t)},N^{-1}}} \;dt
\\
+\frac i{2\pi}\int_0^1 
\RTr\lrpar{ t\lrpar{1-t} \lrspar{N\lrspar{\theta e^{w_N(t)},\theta}, N^{-1}}}\;dt,
\end{multline*}
and we need only apply the trace-defect formula for $\RTr$ explained in
Appendix~\ref{sec:ResTrace},
\begin{equation}
	\RTr\lrpar{\lrspar{A,B}} = 
	- \Trsig{\frac{BD_Q\lrpar A + D_Q\lrpar AB}2} 
	+\Trpa{\frac{BD_x\lrpar A + D_x\lrpar AB}2}.
\label{bcTrDefect0}\end{equation}
Notice that the principal symbol of $\cN$ and the full symbol of $Q^\tau$ are independent of the interval variable $r\in \cI$ and only depend on the cotangent variable, so the formula for the full symbol of the commutator $\lrspar{A,Q\lrpar z}$ shows that this operator is of order $-z-2$ and hence the first term in \eqref{bcTrDefect0} vanishes.
The second term can be written in terms of the indicial families at the $b$ and $c$ ends much like \eqref{BTraceDefect}. Thus, at the $b$-end we get
\begin{equation*}
\frac i{2\pi}\int_0^1 
\lrspar{
\sideset{^R}{}\int_{\bbR} \tr\lrpar{ \theta_\xi e^{w_a(t)}} \;d\xi 
+
\frac i{2\pi} t\lrpar{1-t} \sideset{^R}{}\int_{\bbR} 
	\tr\lrpar{ -\theta_\xi \lrspar{\theta_a e^{w_a(t)}, \theta_a} } \;d\xi } \;dt , 
\end{equation*}
which, as in the proof of Lemma~\ref{lem:dEta}, equals $\int_{\bbR} \df i_{\pa_\xi} \tCh\lrpar{I_b\lrpar\cN}$ and similarly at the cusp end.
This proves the first line in \eqref{Z-i.80}. The second line follows since, on the one hand, we required that $i^*_{\pa X}\Ch\lrpar\bbE$ vanish, and, on the other, the cusp indicial family is given by
\begin{equation*}
	I_c\lrpar{\cN\lrpar{y,\eta}}\lrpar\xi = {}^0\sigma\lrpar A\lrpar{0,y,\xi,\eta}.
\end{equation*}
\end{proof}

With this lemma we have all of the ingredients for representing the Chern character of a relative K-theory class in terms of a zero pseudodifferential operator using the description of the relative cohomology from \S\ref{Zero-ind-coh}. Recall that the chain complex is
\begin{gather*}
	\cZ_k 
	= \CI\lrpar{X;\Lambda^k}
	\oplus 
	\lrpar{ \CI\lrpar{\bbS W;\Lambda^{k-1}}
	\oplus \CI\lrpar{\bar L;\Lambda^{k-1}} }
	\oplus \CI\lrpar{\bbS U;\Lambda^{k-3}} \\
	D = \begin{pmatrix} d & & & \\ -\pi^*_{\bbS W} & -d & & \\ 
		-\pi^*_{\bar L}i^*_\pa & & -d & \\ & -\nu^{\bbS W}_* & \pi^*_{\bbS U}\nu^L_* & d \end{pmatrix}
\end{gather*}
(with a compatibility condition at $L_{\pm}$).

\begin{theorem}\label{thm:ZeroFamilies} $ $

a)
The element of $\cC_{\even}$,
\begin{equation} \label{ZeroChern}
	\Ch\lrpar{W,\bbE,\sigma,\cN}
	= \lrpar{\Ch\bbE, \lrpar{ \tCh a, \tCh I_b\lrpar \cN }, -\eta\lrpar\cN},
\end{equation}
is $D$-closed and represents the Chern character of the K-theory class defined by $\lrpar{W,\bbE,\sigma, \cN}$.

b) Suppose that $A \in \Psi^0_0\lrpar{M/B;\bbE}$ is a fully elliptic family of $0$-pseudodifferential operators acting on the fibers of a fibration $M\to B$, then the Chern character of the index bundle of $A$ is given by
\begin{multline}\label{ZeroFamilies}
	\Ch\lrpar{\Ind A} \\
	= \int_{S^*M/B,S^*\pa M/B} 
	\Ch\lrpar{T^*M/B, \bbE, \sigma\lrpar A, \cN\lrpar A} \wedge \Td\lrpar{M/B} \\
	= \int_{{}^0 S^*M/B} \tCh\lrpar{\sigma\lrpar A}\wedge \Td\lrpar{M/B}
	- \int_{S^*\pa M/B} \eta\lrpar{\cN\lrpar A}\wedge \Td\lrpar{\pa M/B} .
\end{multline}
\end{theorem}

\begin{proof} 
Showing that \eqref{ZeroChern} is $D$-closed is equivalent to showing
\begin{gather*}
	d\Ch\bbE =0, \quad
	d\tCh a = -\pi^*\Ch\bbE, \\
	d\tCh^{ev} I_b\lrpar\cN = 0, \quad
	d\eta\lrpar\cN 
	= -\nu^{\bbS W}_* \tCh(\sigma(A)) + \pi^*_{\bbS U} \nu^{\bar L}_*\tCh(I_b(\cN))
\end{gather*}
the last one of which is Lemma~\ref{Z-i.81} (since $\nu^{\bar L}_*\tCh(I_b(\cN)) = \Ch^\even(I_b(\cN))$) while the other three follow as in the previous section. 

Also as in the previous sections, the forms themselves are invariant under stabilization and bundle isomorphism while Lemma~\ref{HomotopyInv} shows that the cohomology class they define is homotopy invariant. It follows that we can change the $0$-operator whose normal operators define $a$ and $\cN$ as long as we stay in the same $K$-theory class. As explained in \S\ref{sec:Kthy} there is a representative of this class with $\cN$ equal to the identity and $a$ equal to the identity near the boundary.
For this representative the map above clearly coincides with the Chern character and hence this is true for any representative.

In the same way \eqref{ZeroFamilies} follows from homotopy invariance and reduction to the case that the operator is trivial near the boundary.
\end{proof}

\subsection{Transmission families index theorem} \label{FedosovBoutet}

In the same situation as above any class in $\Kc(T^*M^\circ/B)$ can also be
represented by a fully elliptic operator or order and type zero in Boutet
de Monvel's transmission calculus.  With $W=T^*M/B$, $L$ equal to the
normal bundle to the boundary, and $U=W\rest{\pa M}/L$ a K-theory class is
represented by a superbundle over $M$, $\bbE$, a superbundle over $\pa M$,
$\bbF$, and two maps: a family of isomorphisms $a \in \CI\lrpar{\bbS
  W;\hom\lrpar\bbE}$ and a family of Wiener-Hopf operators $N \in
\CI\lrpar{\bbS U; \Psi^*_{WH}\lrpar{L;\lrpar{H^+\otimes\bbE}\oplus\bbF}}$.
We use this data, and a choice of graded connections on $\bbE$ and $\bbF$,
to represent the Chern character using the description of the relative
cohomology from \S\ref{Zero-ind-coh} in terms of forms on $X$, $\bbS W$,
$\pa X$, and $\bbS U$. As in the previous section we will assume that
\begin{quote}
the graded connection on $\bbE$ is chosen so that the restriction of
  $\Ch\lrpar\bbE$ to the boundary is identically zero.
\end{quote}
This description of the Chern character and the resulting index formula are treated by
Fedosov in \cite[$\S$III.4]{MR1401125}; as in Proposition~\ref{Z-i.7}
above, we reinterpret this formula in an appropriate formulation of relative cohomology.

The family of Wiener-Hopf operators plays much the same role as the family of $b,c$ operators in the previous section. 
Were these trace-class it would be natural to use their Chern character in our constructions. However, they are not trace-class so we are forced to use a renormalized trace and we refer to the resulting form as an `eta' form.
In this case the renormalized trace was defined by Fedosov as follows.
With  
$N = \begin{pmatrix} h^+p+b & k \\ t & q \end{pmatrix}$
as in \eqref{BdyBoutet}, 
$N$ is trace-class if and only if $h^+p=0$ so we define $\tr'$ by ignoring this term \cite[\S 4]{MR1401125}
\begin{equation}\label{FedosovTrace}
	\tr'\lrpar N = \tr q + 
	\frac1{2\pi} \int^+_\bbR \tr b\lrpar{\xi,\xi} \;d\xi,
\end{equation}
where in the first term we use $\tr:\hom\lrpar{\bbF}\to\bbR$ applied to $q$, and in the second
$\tr:\hom\lrpar{\bbE}\to\bbR$ applied to the integral kernel of $b$ at the point $\lrpar{\xi,\xi}$.

This renormalized trace is not an actual trace, but instead satisfies a trace-defect formula \cite[Lemma 2.1]{MR1401125}
\begin{equation}\label{FedosovTraceDefect}
	\tr'\lrspar{N_1,N_2} 
	=-\frac i{2\pi}\int_\bbR^+ \tr\lrpar{\frac{\pa p_1\lrpar\xi}{\pa\xi} p_2\lrpar\xi} \;d\xi.
\end{equation}
It is clear from this formula that if either $N_1$ or $N_2$ is `singular' (i.e., $p_1=0$ or $p_2=0$) then $\tr'\lrspar{N_1,N_2}=0$. 

Denote the chosen graded connections on $\bbE$ and $\bbF$ by $\nabla$ and $\nabla^{\pa}$ respectively. Using both of these we define a connection on each of the bundles $\lrpar{H^+\otimes E_\pm}\oplus F_\pm$ acting trivially on $H^+$, we denote the resulting graded connection again by $\nabla$. Notice that $d\tr'\lrpar N = \tr'\lrpar{\nabla N}$.

The trace-defect formula \eqref{FedosovTraceDefect} is formally identical to that of the $b$-trace \eqref{BTraceDefect} on smoothing operators. So if we define $\eta\lrpar N$ as an element of  $\CI\lrpar{\bbS U;\Lambda^\odd}$ by
\begin{multline}\label{FedosovEta}
	\eta\lrpar N= - \frac 1{2\pi i}\int_0^1 
	\lrpar{1-t}\tr'\lrpar{N^{-1}\lrpar{\nabla N}e^{w_N(t)}} 
		+ t\tr'\lrpar{\lrpar{\nabla N}e^{w_N(t)}N^{-1}} \;dt\\
	\Mwhere 
	w_N(t)=(1-t)\omega _++tN^{-1}\omega _-N+\frac1{2\pi i}t(1-t)(N^{-1}\nabla N)^2\\
\end{multline}
then the computations in the proof of Lemma \ref{lem:dEta} apply verbatim to compute $d\eta\lrpar N$ and we conclude that 
\begin{equation}\label{dEtaFedosov}
	d\eta\lrpar N
	= -\pi^*_\pa\Ch\lrpar{\bbE_\pa \oplus \bbF} + \nu_*^L i^*_\pa \tCh\lrpar a,
\end{equation}
where, with $-2\pi i \omega'_{\pm}$ equal to the curvature of $\nabla$ on $\lrpar{H^+\otimes \bbE}\oplus \bbF$,
\begin{equation*}
	\Ch\lrpar{ \bbE_\pa\oplus\bbF} 
	= \tr e^{\omega'_+}-\tr^{\omega'_-} \in \CI\lrpar{\pa X;\Lambda^\even}
\end{equation*}
and as usual
\begin{multline}
	\tCh\lrpar a=-\frac 1{2\pi i}\int_0^1 
	\tr\left(a^{-1}(\nabla a)e^{w_a(t)}\right)dt 
		\in \CI\lrpar{\bbS W;\Lambda ^\odd},\\
	\Mwith w_a(t)
	=(1-t)\omega _++ta^{-1}\omega _-a+\frac 1{2\pi i}t(1-t)(a^{-1}\nabla a)^2.
\end{multline}

It follows that with the chain space from \S\ref{Zero-ind-coh}
\begin{gather*}
\cT^k = \CI(X;\Lambda^k)
\oplus
\lrpar{ \CI_{\pm}(\bbS W;\Lambda^{k-1}) \oplus \CI(\pa X; \Lambda^{k-2}) }
\oplus
\CI(\bbS U;\Lambda^{k-3}) \\
	D_{\cT} = \begin{pmatrix} d & & \\ \phi_1 & -d & \\ & \phi_2 & d \end{pmatrix},
\phi_1 = \begin{pmatrix} -\pi^*_{\bbS W} \\ i^*_{\pa X} \end{pmatrix},
\phi_2 = \begin{pmatrix} \nu^{\bbS W}_*, \pi^*_{\bbS U} \end{pmatrix},
\end{gather*}
the forms
\begin{equation}\label{TransChern}
	\Ch\lrpar{\bbE,\bbF,a,N}
	=\lrpar{ \Ch\lrpar \bbE, \lrpar{ \tCh\lrpar a, -\Ch\lrpar{\bbE_\pa\oplus\bbF} }, -\eta\lrpar N}
\end{equation}
define a relative cohomology class.

\begin{theorem}\label{thm:TransFamilies} $ $

a) $\Ch\lrpar{\bbE,\bbF,a,N}$ is a $D$-closed element of $\cT^{\even}$ and
represents the Chern character of the K-theory class defined by $\lrpar{W,
  \bbE, \bbF, a, N}$.

b) Suppose that $\cA \in \Psi^0_{\tm}\lrpar{M/B;\bbE;\bbF}$ is a fully
elliptic family of transmission pseudodifferential operators acting on the
fibers of a fibration $M\to B$, then the Chern character of the index
bundle of $\cA$ is given by
\begin{multline}\label{TransFamilies}
	\Ch (\Ind (\cA) ) \\
	= \int_{S^*M/B,S^*\pa M/B} 
	\Ch\lrpar{T^*M/B, \bbE, \bbF, \sigma\lrpar\cA, N\lrpar\cA}
	\wedge \Td\lrpar{X,\pa X}\\
	=\int_{S^*M/B} \tCh\lrpar{\sigma\lrpar\cA}\wedge  \Td\lrpar{M/B} 
	- \int_{S^*\pa M/B} \eta\lrpar{N\lrpar\cA} \wedge \Td\lrpar{\pa M/B}.
\end{multline}
\end{theorem}

\begin{proof} 
As in the previous section, we have shown that \eqref{TransChern} defines a
$D$-closed form, Lemma~\ref{HomotopyInv} shows that it depends only on the
K-theory class and evaluating it on a representative that is trivial at the
boundary shows that \eqref{TransChern} is the usual Chern character and
reduces \eqref{TransFamilies} to the Atiyah-Singer families index theorem.
\end{proof}

\appendix

\section{Normal operator}\label{NormalOp}

We review the reduced normal operator in the zero calculus. 

Zero differential operators are elements of the enveloping algebra of the
vector fields that vanish at the boundary, $\cV_0$.  Thus if $x$ is a
boundary defining function and $y_i$ are local coordinates along the
boundary of $X$, we have
\begin{gather*}
\cV_0 = \mathrm{Span}_{\CI(X)}\ang{ x\pa_x, x\pa_{y_1}, \ldots, x\pa_{y_n} },
\\
\Mand
P \in \mathrm{Diff}_0^k(M;E,F) \iff
P = \sum_{j+|\alpha| \leq k} a_{j,\alpha}(x,y) (x\pa_x)^j (x\pa_y)^\alpha,
\end{gather*}
where  the coefficients,  $a_{j,\alpha}$, are  section of  the homomorphism
bundle $\hom(E,F)$.

It is convenient to study these operators by viewing their Schwartz kernels
as defined on a compactification of $(M^\circ)^2$ different from $\bar
M^2$.  We construct $M^2_0$ from $\bar M^2$ by replacing the diagonal at
the boundary, $\diag_{\pa M}$, with its inward pointing spherical normal
bundle,
\begin{equation*}
	M^2_0 = 
	(\bar M^2 \setminus \diag_{\pa M})
	\bigsqcup
	(S^+\diag_{\pa M})
\end{equation*}
and endowing the result with the smallest smooth structure including smooth
functions on $\bar M^2$ and polar coordinates around $\diag_{\pa M}$ -- a
process known as `blowing-up $\diag_{\pa M}$ in $\bar M^2$'.  The diagonal
lifts from the interior of $M^2$ to a submanifold, $\diag_0$, of $M^2_0$ (a
$p$-submanifold in the sense of \cite{Corners}).  The set of `Dirac
sections' along the diagonal on $M^2$ coincides with the space of Schwartz
kernels of arbitrary differential operators on $M$, the Schwartz kernels of
zero differential operators are precisely those that lift to Dirac sections
of $M^2_0$ along $\diag_0$.

The passage from zero differential operators to zero pseudodifferential
operators is a micro-localization in that we enlarge the space of Schwartz
kernels from Dirac sections along $\diag_0$ to distributions on $M^2_0$
with conormal singularities along $\diag_0$, i.e., $I^s(M^2_0, \diag_0)$ in
the notation of \cite{Corners}. We further demand that at each boundary
face of $M^2_0$ the Schwartz kernel have a classical asymptotic expansion
in the corresponding boundary defining function and its logarithm. We refer
to \cite{Mazzeo:Hodge}, \cite{Mazzeo-Melrose} for the details.

We can use the expansions at each of the boundary faces to define `normal
operators' by restricting to the leading term at that face. The zero
diagonal meets the boundary at the front face (the face introduced by the
blow-up of $\diag_{\pa M}$) and the normal operator at this face is known
as {\em the} normal operator. Directly from the definition it is clear that
the restriction of the kernel of a zero pseudodifferential operator of
order $s$ to the front face is a distribution in
\begin{equation*}
I^{s+\frac14}(S^+\diag_{\pa M}, \diag_0 \cap S^+\diag_{\pa M})
\cong
I^{s+\frac14}( \pa M \times \bbR^{m-1} \times \bbR^+, \pa M \times \{(0, 1)\}).
\end{equation*}
It can be shown that the kernels of normal operators define, for each $q
\in \pa M$, a distribution on $T_qM^+ \cong \bbR^{m-1} \times \bbR^+$ that
is invariant with respect to the affine group action
\begin{equation*}
	\xymatrix @R=1pt
	{ (\bbR^{m-1}\times \bbR^+)^2 \ar[r] & \bbR^{m-1} \times \bbR^+ \\
	( (a,b), (c,d) ) \ar@{|->}[r] & (a+bc, bd) } 
\end{equation*}
and that the normal operator is a homomorphism in that the kernel of the
normal operator of the composition of two operators is the convolution of
the normal operators. Thus analyzing the normal operator is tantamount to
harmonic analysis of the affine group on $\bbR^{m-1} \times \bbR^+$.

The invariance with respect to the affine group action suggests studying
the normal operator by first taking the Fourier transform in $\bbR^{m-1}$
and then exploiting dilation invariance with respect to the $\bbR^+$
variable. The result is a family of operators acting on the normal bundle
to the boundary and parametrized by the cosphere bundle to the
boundary. The choice of a trivialization of the normal bundle allows us to
identify this with an element of
\begin{equation*}
	\CI( \bbS^* \pa M, \Psi^s_b(\bbR^+) ).
\end{equation*}

Finally let $x$ be the identity map $\bbR^+ \to \bbR^+$, we compactify
$\bbR^+$ by introducing $\frac1x$ as a boundary defining function for the
`point at $\infty$'.  After performing these transformations on the normal
operator, and identifying the compactification of $[0,\infty)$ with
  $[0,1]$, the normal operator is an operator of $b$-type at the $0$ end
  and an operator of $sc$-type at the $1$ end, with a singularity of order
  $-s$. This family is known as the {\em reduced normal operator}, \eg,
\begin{equation*} A \in \Psi^0_0(X) \implies \cN(A) \in \CI(\bbS^* \pa
  M, \Psi^0_{\fbsc}([0,1]) ).
\end{equation*}

The distributional kernels of elements of the $\fbsc$ can be conveniently
described on the $\fbsc$ double space, $[0,1]^2_{\fbsc}$ obtained from
$[0,1]^2$ in two steps. In the first, we blow-up the corners $\{(0,0),
(1,1)\}$ obtaining two new boundary faces, one $bf_0$ replacing $\{(0,0)\}$
and one $bf_1$ replacing $\{(1,1)\}$. In the second step, we blow-up the
intersection of $bf_1$ with the closure of the diagonal of $(0,1)^2$; we
refer to the resulting boundary face as $cf$. The closure of the diagonal
of $(0,1)^2$ in $[0,1]^2_{\fbsc}$ is referred to as
$\diag_{\fbsc}$. Pseudo-differential operators of $\fbsc$ type on $[0,1]$
are those whose Schwartz kernels are distributions on $[0,1]^2_{\fbsc}$
conormal with respect to $\diag_{\fbsc}$ and otherwise smooth that vanish
to infinite order at every boundary face that does not meet
$\diag_{\fbsc}$.

These operators are sometimes known as the {\em small} calculus of $\fbsc$
operators and there is also an associated {\em large} calculus of $\fbsc$
pseudodifferential operators where the Schwartz kernels need not vanish at
the boundary faces that share a corner with $bf_0$ (known as the `side faces'). It is standard (see
\cite{APSBook}, \cite{Mazzeo-Melrose}) that parametrices and generalized inverses of $\fbsc$
pseudodifferential operators, when these exist, are elements of the large
calculus of $\fbsc$ operators.

A complete metric on the interior of $[0,1]_x$ is of $\fbsc$ type if it takes the form $\frac{dx^2}{x^2}$ near $x =0$ and the form $\frac{dx^2}{(1-x)^4}$ near $x=1$. The corresponding space of square-integrable functions is denoted $L^2_{\fbsc}([0,1])$.
Elements of the small calculus of $\fbsc$ operators of order zero define bounded operators on $L^2_{\fbsc}([0,1])$ as do elements of the large calculus, so long as they vanish to order greater than $-1$ at each of the side faces.
Thus it makes sense to compose these operators and the references cited above show that the composition is given again by an element of, respectively, the small or large $\fbsc$ calculus.

\section{Residue traces}\label{sec:ResTrace}

We review the residue traces on algebras of pseudodifferential
operators. These are used to define Chern and eta forms. We start by
presenting the case of closed manifolds, before moving on to manifolds with
boundary and finally operator valued forms. We refer the reader to
\cite{Melrose-Nistor0} for the proofs of these statements for the algebra
of cusp operators. The proofs consist of formal algebraic manipulations and
hence hold verbatim for many other calculi -- we shall need these results
for the $b, sc$ algebra in section \ref{sub:ZeroFIT}

On a closed manifold, $X$, the Guillemin-Wodzicki residue trace, $\tTrsig$,
is the unique trace on $\Psi^*\lrpar X$, and has two well-known expressions
\begin{equation}\label{WodRes}
\Trsig A=\int_{S^*X} \sigma_{(-n)}\lrpar A = \Res_{\tau=0}\Tr\lrpar{Q^{-\tau}A}.
\end{equation}
The first expression (due to Wodzicki) involves the term in the expansion
of the full symbol of $A\in\Psi^\bbZ\lrpar X$ that is homogeneous of order
$-n$, $\sigma_{(-n)}$. Although $\sigma_{(-n)}\lrpar A$ is not invariantly
defined, its integral is (notice that scale invariance as a density forces
the homogeneity). The second expression (due to Guillemin) involves the
choice of an {\em admissible} first order operator $Q$, which we shall
understand to mean that $Q$ is invertible, essentially self-adjoint, and
positive.  The function $\tau\mapsto \Tr\lrpar{Q^{-\tau}A}$ is holomorphic
on a half-plane and can be meromorphically extended to the whole complex
plane. The pole at $\tau=0$ is at worst a simple pole and the residue is
independent of the choice of $Q$.

The first expression above shows that the Guillemin-Wodzicki residue
vanishes on operators of order less than $-n$, the space of trace-class
operators. For each choice of admissible operator $Q$, we can extend the
trace from trace-class operators to a {\em renormalized} trace on all of
$\Psi^*$ by
\begin{equation}\label{ReWod}
	{}^R\Tr\lrpar A = \FP_{\tau=0}\Tr\lrpar{Q^{-\tau}A} 
		= \lim_{\tau\to0}\lrpar{\Tr\lrpar{Q^{-\tau}A}-\frac1\tau\hat\Tr^\sigma\lrpar A}.
\end{equation}
This renormalized trace does depend on the choice of $Q$, indeed
\begin{equation*}
	{}^R\Tr_{Q_1}\lrpar A - {}^R\Tr_{Q_2}\lrpar A = - \hat{\Tr^\sigma}\lrpar{A\log\lrpar{Q_1/Q_2}}
\end{equation*}
where we define 
\begin{equation}\label{logQ}
	\log\lrpar{Q_1/Q_2} = \dzero{\tau}\lrpar{Q_1^\tau Q_2^{-\tau}}.
\end{equation}
We also point out that ${}^R\Tr$ is not really a trace, since it does not vanish on commutators:
\begin{equation*}
	{}^R\Tr\lrpar{\lrspar{A,B}} = -\hat{\Tr^\sigma}\lrpar{D_Q\lrpar AB},
\end{equation*}
with $D_Q$ the derivation on $\Psi^*\lrpar X$ defined by
\begin{equation*}
	D_Q\lrpar A = \dzero{\tau}Q^\tau AQ^{-\tau}.
\end{equation*}

In this subsection we review the extension by Melrose-Nistor
\cite{Melrose-Nistor0} (see also \cite{LauterMoroianu},
\cite{MoroianuNistor}) of the Guillemin-Wodzicki residue trace to operator
algebras on manifolds with boundary.

We define, for each choice of bdf $x$ and admissible first order positive
$Q\in\Psi^1_{\fbsc}\lrpar{M;E}$, `trace' functionals on the algebra
$\Psi^*_{\fbsc}\lrpar{M;E}$. By a simple extension of the results of Piazza
\cite{Piazza}, complex powers of $Q$ can be defined as elements of the {\em
  large} calculus of $\fbsc$ operators (described in \S\ref{NormalOp}).
Assume for the moment that $A$ and $B$ are elements of the small calculus
of $\fbsc$ operators.  It is shown in \cite[Lemma 4]{Melrose-Nistor0} that
the function
\begin{equation*}
\cZ\lrpar{A;z,\tau} =
\frac12\lrspar{\Tr\lrpar{x^zQ^{-\tau}A}+\Tr\lrpar{Q^{-\tau}x^zA}}
\end{equation*}
is holomorphic for $\Re{z}, \Re{\tau} >>0$ and extends meromorphically to
$\bbC^2$ with at most simple poles in $z$ or $\tau$. We use the expansion
at zero to define the following functionals
\begin{equation}\label{TraceDef}
\cZ\lrpar{A;z,\tau} \sim
\frac1{z\tau}\Trpasig A + \frac1z\Trpa A + \frac1\tau\Trsig A + \RTr\lrpar A
+ \cO\lrpar{ z,\tau}.
\end{equation}

Among these functionals, only $\tTrpasig$ is a trace (vanishes on all
commutators) and is independent of the choice of $x$ and $Q.$  We note that
the residue trace functionals could be defined using
$\Tr\lrpar{x^zQ^{-\tau}A}$ or $\Tr\lrpar{Q^{-\tau}x^zA}.$ This follows from
the fact that
\begin{equation*}
\Tr\lrpar{x^zQ^{-\tau}A}-\Tr\lrpar{Q^{-\tau}x^zA}
\end{equation*}
is regular at $z=0,$ $\tau=0.$ The renormalized trace, though, would {\em a
  priori} be different.

On the ideal $x^{\infty}\Psi^*_{\fbsc}\lrpar X$, the functionals
$\tTrpasig$ and $\tTrpa$ both vanish, the functional $\tTrsig$ is a trace
and is given by the expressions \eqref{WodRes}.  Similarly, on the ideal
$\Psi^{-\infty}_{\fbsc}\lrpar X$, the functionals $\tTrpasig$ and
$\tTrsig$ both vanish, the functional $\tTrpa$ is a trace and has an
expression similar to \eqref{WodRes}. Indeed, recall that, by Lidskii's
theorem, if $A$ is a trace-class operator with Schwartz kernel $\cK_A$,
then
\begin{equation*}
	\Tr\lrpar A = \int_X \cK_A\rest{\diag}
\end{equation*}
(note that $\cK\rest{\diag}$ is a density).
For a general operator $A\in\Psi^{-\infty}_{\fbsc}\lrpar X$, having chosen a bdf $x$, we can expand $\cK\rest{\diag}$ in its `Taylor's series' at $x=0$
\begin{equation}\label{Taylor}
	\cK_A\sim \sum_{\ell\geq0}x^\ell\cK_\ell,
\end{equation}
and one of these terms will be invariant under the rescaling
$x\mapsto\lambda x$, say $x^r\cK_r$ (since $\cK_A$ is a singular density at
$x=0$, this is not $\cK_0$).  It is easy to see that the residue at $z=0$
of
\begin{equation*}
	\Tr\lrpar{x^z A} = \int_X x^z\cK_A\rest{\diag}
\end{equation*}
is given by the partial integral of $\cK_r$,
\begin{equation*}
	\Trpa A = \int_{\pa X} \cK_r\rest{\pa X}.
\end{equation*}
Just as for the Guillemin-Wodzicki residue, the density $\cK_r\rest{\pa X}$
depends on the choice of $x$, but its integral along the boundary does not.

For operators not in these ideals, $\tTrsig$ is obtained by taking a
`residue in the symbol' and `renormalization at the boundary' meaning, for
instance, that we renormalizw the integral occuring in the first expression
in \eqref{WodRes}.  Where by renormalized integral of a density $\mu$, we
mean
\begin{equation*}
	\sideset{^R}{}\int \mu = \FP_{z=0} \int x^z\mu.
\end{equation*}
Similarly, $\tTrpa$ is obtained by taking a `residue at the boundary' and
`renormalization in the symbol'. So, for instance, expand the operator's
kernel as in \eqref{Taylor} as a distribution to pick out $\cK_r$ (residue
at the boundary), this defines an pseudodifferential operator over the
boundary and then its renormalized trace over the boundary defined as in
\eqref{ReWod} equals $\tTrpa$.  The common residue, i.e., the `residue at
the boundary' and `residue in the symbol', is given by the functional
$\tTrpasig$.

The functional $\tTrsig$ is independent of $Q$ just as in the case of a
closed manifold, but does depend on the choice of $x$ via \cite[Lemma
9]{Melrose-Nistor0}
\begin{equation*}
	\Trsig[x_1] A-\Trsig[x_2] A
	= \Trpasig{ \log\lrpar{x'/x} A}.
\end{equation*}
Similarly, $\tTrpa$ is independent of the choice of $x$, but depends on the choice of $Q$ via \cite[Lemma 11]{Melrose-Nistor0}
\begin{equation*}
	\Trpa[Q_1] A - \Trpa[Q_2] A = - \Trpasig{\log\lrpar{Q_1/Q_2}A}
\end{equation*}
with $\log\lrpar{Q_1/Q_2}$ defined by \eqref{logQ}.
It follows that the functional $\tTrpasig$ is independent of both the choice of $x$ and the choice of $Q$.

To analyze the behavior of these functionals on commutators, notice that
\begin{equation*}\begin{split}
	x^zQ^{-\tau}\lrspar{A,B}
	&= x^zQ^{-\tau}\lrpar{A-Q^\tau AQ^{-\tau}}B
	\\&\phantom{xxx}
	- x^zQ^{-\tau} B\lrpar{A-x^z A x^{-z}}
	+\lrspar{x^z A x^{-z}, x^zQ^{-\tau}B} \\
	&= -x^zQ^{-\tau}\lrpar{\tau D_Q\lrpar A + \cO\lrpar{\tau^2}}B
	\\&\phantom{xxx}
	+ x^zQ^{-\tau}B \lrpar{z D_x\lrpar A + \cO\lrpar{z^2}}
	+\lrspar{x^z A x^{-z}, x^zQ^{-\tau}B} \\
\end{split}\end{equation*}
where
\begin{equation*}\begin{split}
	D_x\lrpar{A} &= \dzero{z}\lrpar{x^zAx^{-z}} =: \lrspar{\log x, A},\\
	D_Q\lrpar{A} &= \dzero{\tau}\lrpar{Q^ \tau AQ^{-\tau}} =: \lrspar{\log Q, A}.
\end{split}\end{equation*}
and since $\Tr\lrpar{x^zQ^{-\tau}A}$ has only simple poles in $z$ and $\tau$ the terms with $\cO\lrpar{\tau^2}$ and $\cO\lrpar{z^2}$ do not affect the finite part at zero, i.e. $\Tr\lrpar{x^zQ^{-\tau}\lrspar{A,B}}$ has the same finite part at zero as $\Tr\lrpar{x^zQ^{-\tau}\lrpar{-\tau D_Q\lrpar A B + zBD_x\lrpar A}}$. 

A similar computation shows that $\Tr\lrpar{Q^{-\tau}x^z\lrspar{A,B}}$ has the same finite part at zero as
$\Tr\lrpar{Q^{-\tau}x^z\lrpar{zD_x\lrpar A B - \tau BD_Q\lrpar A}}$. 
This shows that the expansion of $2\cZ\lrpar{\lrspar{A,B};z,\tau}$ at $z=0,\tau=0$ is
\begin{multline}\label{TraceComm}
	-\frac2z\Trpasig{BD_Q\lrpar A}
	+\frac2\tau\Trpasig{BD_x\lrpar A}\\
	- \Trsig{BD_Q\lrpar A + D_Q\lrpar AB} +\Trpa{BD_x\lrpar A + D_x\lrpar AB} \\
	- \frac\tau{z}\Trpa{BD_Q\lrpar A + D_Q\lrpar AB} + \frac z\tau\Trsig{BD_x\lrpar A + D_x\lrpar AB}
	+ \cO\lrpar{\tau,z}.
\end{multline}
Notice that if $\lrspar{A,B}$ is in $\Psi^{-\infty}_{\fbsc}\lrpar{M;E}$
then those terms in \eqref{TraceComm} with $\frac1z$ must vanish, while if
it is in $x^\infty\Psi^*_{\fbsc}\lrpar{M;E}$ then those terms with
$\frac1\tau$ must vanish. This observation can be used, together with
Calderon's formula for the index, to obtain a residue trace formula for the
index of a Fredholm operator, see \cite{Melrose-Nistor0},
\cite{MoroianuNistor}.

For any $\lambda>1$ notice that
\begin{equation*}
	\hat\cZ\lrpar{A;z}
	= \frac{\cZ\lrpar{A;z,\lambda z} + \cZ\lrpar{A;z,\frac1\lambda z} 
		- \lrpar{\lambda+ \frac1\lambda}\cZ\lrpar{A;z,z}}
	{2-\lrpar{\lambda+\frac1\lambda}}
\end{equation*}
coincides with the usual trace if $A$ is trace-class, and from \eqref{TraceComm} satisfies
\begin{multline}\label{NewCommTrace}
	\hat\cZ\lrpar{\lrspar{A,B};z}
	\sim 
	-\frac1z\Trpasig{BD_Q\lrpar A} \\
	- \Trsig{\frac{BD_Q\lrpar A + D_Q\lrpar AB}2} 
	+\Trpa{\frac{BD_x\lrpar A + D_x\lrpar AB}2}
	+ \cO\lrpar z,
\end{multline}
near $z=0$.
Thus the renormalized trace of a commutator is given by the second line in \eqref{NewCommTrace} at $z=0$,
\begin{equation}\label{TrDefect}
	\RTr\lrpar{\lrspar{A,B}} = 
	- \Trsig{\frac{BD_Q\lrpar A + D_Q\lrpar AB}2} 
	+\Trpa{\frac{BD_x\lrpar A + D_x\lrpar AB}2}.
\end{equation}

Finally we point out that the same formula holds for elements of the {\em large} calculus of $\fbsc$ operators. Indeed, we can choose a sequence of smooth, non-negative functions $\phi_k \in \CI([0,1]^2_{\fbsc})$ that are identically equal to one in a neighborhood of the lifted diagonal $\diag_{\fbsc}$, vanish to infinite order at any boundary face that does not meet $\diag_{\fbsc}$, converge uniformly to the constant function $1$, and, in a collar neighborhood of the boundary face $bf_0$ are independent of the boundary defining function. Denote the right-hand-side of \eqref{TrDefect} by $\cF(A,B)$ and define the operators $A_k$ and $B_k$ by multiplying the distributional kernels of $A$ and $B$ respectively by $\phi_k$.
Then $A_k$ and $B_k$ are in the small calculus and we point out that
\begin{equation*}
	\RTr([A,B]) 
	= \lim_k \RTr( [A_k, B_k])
	= \lim_k \cF(A_k, B_k)
	= \cF(A, B),
\end{equation*}
since, for instance, the symbol of $A_k$ coincides with that of $A$ and the expansion of $A_k$ at $bf_0$ is the same as that of $A$ with the coefficients multiplied by $\phi_k$.

\section{Operator valued forms}\label{Op-forms}

We will make use of forms on the vector space $\Psi^*_{\fbsc}$ pulled back
to $\bbS^*Y$ with values in $\cG^0_{b,sc}$, the invertible elements of
order zero in the $b,sc$ calculus. It will be useful to have the analogue
of the trace defect formula \eqref{TrDefect}.

Consider two pure forms with values in $\Psi^0_{b,sc}$, say
\begin{equation*}\begin{split}
	\eta \in \CI\lrpar{\bbS^*Y,\Lambda^k\Psi^0_{b,sc}},& \phantom{xx}
		\eta = A \hat\eta \text{ with }\hat\eta\in\Omega^k\bbS^*Y\\
	\omega \in \CI\lrpar{\bbS^*Y,\Lambda^\ell\Psi^0_{b,sc}},& \phantom{xx}
	\omega = B \hat\omega \text{ with }\hat\omega\in\Omega^\ell\bbS^*Y
\end{split}\end{equation*}
Notice that
\begin{equation}\label{sbracket}\begin{split}
	\lrspar{\eta,\omega}_s &= \eta\wedge\omega - \lrpar{-1}^{|\eta|\cdot|\omega|}\omega\wedge\eta\\
	&=\lrspar{A,B}\hat\eta\wedge\hat\omega.
\end{split}\end{equation}

The trace functionals and derivations discussed in Appendix \ref{sec:ResTrace} have natural extensions to forms, e.g.,
\begin{equation*}
	\RTr:\CI\lrpar{\bbS^*Y,\Lambda^k\Psi^0_{b,sc}} \to \Omega^k\bbS^*Y,
\end{equation*}
by acting on coefficients (thus $\RTr\lrpar \eta = \RTr\lrpar A \hat\eta$). The trace defect formulas generalize to this context; in particular
\begin{multline}\label{TrDefectForms}
	\RTr\lrpar{\lrspar{\eta,\omega}_s} =
	- \Trsig{\frac{D_Q\lrpar \eta \wedge\omega 
		+ \lrpar{-1}^{|\eta|\cdot|\omega|}\eta \wedge D_Q\lrpar \omega}2} \\
	+\Trpa{\frac{D_x\lrpar \eta \wedge\omega 
		+ \lrpar{-1}^{|\eta|\cdot|\omega|}\eta \wedge D_x\lrpar \omega}2} 
\end{multline}

\end{document}